\documentclass[reqno]{amsart}

\usepackage{amsmath}
\usepackage{amssymb}
\usepackage{amsthm}
\usepackage{enumerate}
\usepackage{verbatim}
\usepackage{graphicx,color}

\title{Application of a unified Kenmotsu-type formula for surfaces 
in Euclidean or Lorentzian three-space %
}
\author{Masatoshi Kokubu}
\date{November 15, 2017}
\address{%
   Department of Mathematics,
   School of Engineering,
   Tokyo Denki University,
   Tokyo, 120-8551,
   Japan
}
\email{kokubu@cck.dendai.ac.jp}
\thanks{%
The author was 
supported by Grant-in-Aid for 
Scientific Research (C) No.~17K05227
 from the Japan Society for the Promotion of Science.}

\renewcommand\L{{\mathbb L}}

\newcommand\E{{\mathbb E}}

\newcommand\la{\langle}
\newcommand\ra{\rangle}
\newcommand\iii{\mathrm{i}}
\newcommand\jjj{\mathrm{j}}
\newcommand\fff{\, \mathrm{I}}
\newcommand\sff{\, \mathrm{I\!I}}
\newcommand\tff{\, \mathrm{I\!I\!I}}

\newcommand\R{{\mathbb R}}
\newcommand\C{{\mathbb C}}

\newcommand\xbf{{\mathbf x}}
\newcommand\ybf{{\boldsymbol y}}
\newcommand\nbf{\boldsymbol n}
\newcommand\cbf{\boldsymbol c}

\renewcommand\Re{\operatorname{Re}}
\renewcommand\Im{\operatorname{Im}}

\newcommand\sn{\mathop{\mathrm{sn}}\nolimits}
\newcommand\cn{\mathop{\mathrm{cn}}\nolimits}
\newcommand\dn{\mathop{\mathrm{dn}}\nolimits}
\newcommand\am{\mathop{\mathrm{am}}\nolimits}

\newcommand\ns{\mathop{\mathrm{ns}}\nolimits}
\newcommand\cs{\mathop{\mathrm{cs}}\nolimits}

\theoremstyle{plain}
\newtheorem{theorem}{Theorem}[section]
\newtheorem{proposition}[theorem]{Proposition}
\newtheorem{corollary}[theorem]{Corollary}
\newtheorem{lemma}[theorem]{Lemma}

\theoremstyle{definition}

\newtheorem{remark}[theorem]{Remark}
\newtheorem{example}[theorem]{Example}

\numberwithin{equation}{section}

\begin{document}
\begin{abstract}
Kenmotsu's formula describes surfaces in Euclidean 3-space 
by their mean curvature functions and Gauss maps. 
In Lorentzian 3-space, 
Akuta\-gawa-Nishikawa's formula and Magid's formula are Kenmotsu-type formulas 
for spacelike surfaces and for timelike surfaces, respectively. 
We apply them to a few problems concerning rotational or helicoidal surfaces 
with constant mean curvature. 
Before that,  
we show that the three formulas above can be written in a unified single equation.
\end{abstract}

\maketitle

\section{Introduction}
In surface theory, the Weierstrass representation for minimal surfaces 
is one of the fundamental tools. 
There are also Weierstrass-type representation formulas 
for spacelike maximal surfaces and for timelike minimal surfaces 
in Lorentzian $3$-space $\L^3$, and they play an important role as well. 
On the other hand, in the case where surfaces are not minimal, 
Kenmotsu \cite{K} gave a formula which describes surfaces in Euclidean $3$-space 
$\E^3$ in terms of the Gauss map and mean curvature function $H$ 
(under the assumption that $H \ne 0$). After that, 
Akutagawa-Nishikawa \cite{AN} and Magid \cite{M} (see also \cite{AEG})
gave such formulas for 
spacelike surfaces and timelike surfaces in $\L^3$, respectively. 
These three formulas, although established independently, 
have created a common understanding that  
they all come from essentially the same principle. 
One of the aims of this paper is to unify these three formulas 
(see Sections \ref{sec:Ken-formula}, \ref{sec:pf-ken-type}).    

Kenmotsu's, Akutagawa-Nishikawa's and Magid's formulas have been considered as 
important ones in Euclidean or Lorentzian surface theory 
similar to the Weierstrass representation, whereas 
it seems that they do not have much direct application so far. 
Another aim is to present a practical aspect of these three formulas. 
We give applications of a Kenmotsu-type formula to surfaces of constant mean curvature  
from the classical differential geometric viewpoint.   
Although there have been several modern theories 
 nowadays, our argument is independent of them. 
More precisely, we provide representations of (Euclidean or Lorentzian) 
Delaunay surfaces 
and helicoidal surfaces with constant mean curvature ({\it cmc-$H$ helicoids}, 
for short) in $\E^3$ as explicitly as possible
(see Sections \ref{sec:delaunay}, \ref{sec:cmc-helicoids}). 
It has been  pointed out in many literatures that 
Delaunay surfaces can be described 
in terms of elliptic functions and elliptic integrals. 
In fact, for example, we can find in \cite{DHMV}
an explicit description of unduloids. 
We provide in this paper all (Lorentzian) Delaunay surfaces in terms of 
Jacobi's elliptic functions.
We show that cmc-$H$ helicoids can also be expressed with Jacobi's elliptic functions. 
Jacobi's elliptic functions play an important role in this paper.

The remaining aim of this paper is to give further applications 
(see Section \ref{sec:f-i-cmc-heli}).  
To make the most of the explicit representation here, 
we solve the periodicity condition when a cmc-$H$ helicoid is cylindrical, 
and as a result, we give another proof of Burstall-Kilian's theorem \cite{BK}. 
Also, we 
introduce the radius of a cmc-$H$ helicoid, and then
we notice that a cmc-$H$ helicoid is determined by the pitch and radius. 
As a consequence, the periodicity of a cmc-$H$ helicoid is determined by    
the pitch and radius. 
We also give a criterion by the pitch and radius whether two cmc-$H$ helicoids 
belong to the same associated family (i.e. 
a family of non-congruent but locally isometric surfaces with the same mean curvature).  
As a corollary, we can show that an unduloid and a nodoid 
whose mean curvatures are the same value are associated 
(i.e., locally isometric) if and only if the ratios $\rho/R$
of the inner radius $\rho$ and the outer radius $R$ are coincident. 
 

\section{Kenmotsu-type formula}\label{sec:Ken-formula}
We give a coordinate-free formula  
which describes surfaces by their mean curvature functions and Gauss maps.  
As a consequence, formulas due to Kenmotsu, Akutagawa-Nishikawa and Magid can 
be considered within a unified single formula. 
(The original formulas \cite{K}, \cite{AN}, \cite{M} are given in terms of 
local coordinate systems, respectively.)

In fact, our statement is as follows:

\begin{theorem}[cf. \cite{K}, \cite{AN}, \cite{M}]\label{thm:ken-theorem}
Let $M$ be a connected, oriented $2$-dimensional manifold, 
and suppose that $\xbf \colon M \to N$ is either 
\begin{enumerate}[(i)]
 \item an immersion to $N=\E^3=(\R^3, \la , \ra_E)$, 
 \item a spacelike immersion to $N=\L^3=(\R^3, \la , \ra_L)$, or
 \item a timelike immersion to $N=\L^3=(\R^3, \la , \ra_L)$
\end{enumerate}
 with non-zero mean curvature function $H$ and Gauss map $\nbf$. 
Then $\xbf$ can be represented by $H$ and $\nbf$, as 
\begin{equation}\label{eq:unified-formula}
 \xbf = -\int \frac{1}{2H} \left\{ d \nbf + \nbf \times (*d \nbf) \right\}
\end{equation}  
where $\times$ denotes the Euclidean (resp. Lorentzian) vector product for (i) 
(resp. for (ii), (iii)), and $*$ denotes the Riemannian (resp. Lorentzian) 
Hodge $*$-operator on $T^*M$ with respect to the induced metric 
for (i), (ii) (resp. for (iii)).

Conversely, suppose that $\nbf$ is a unit vector-valued function and $H$ is a function
 on a simply-connected Riemannian or Lorentzian $2$-manifold $M$ satisfying 
\begin{equation}\label{eq:co-int-condition}
 d \left(
\frac{ d \nbf + \nbf \times (*d \nbf)}{H}
\right)=0 \quad (d \nbf + \nbf \times (* d \nbf) \ne 0). 
\end{equation}
Then \eqref{eq:unified-formula} gives  
\begin{enumerate}[(i)]
 \item a surface in $\E^3$ if $M$ is Riemannian,  
$\nbf$ is unit Euclidean vector-valued and $\times = \times_E$, 
 \item a spacelike surface in $\L^3$ if $M$ is Riemannian, 
$\nbf$ is negative unit Lorentzian vector-valued and $\times = \times_L$, or
 \item a timelike surface in $\L^3$  if $M$ is Lorentzian,
$\nbf$ is positive unit Lorentzian vector-valued and $\times = \times_L$, 
\end{enumerate}
whose Gauss map and mean curvature function are $\nbf$ and $H$.  
\end{theorem}
In the statement of Theorem \ref{thm:ken-theorem}, the following are utilized:  
\begin{itemize}
 \item 
For $\boldsymbol a={}^t(a_1,a_2,a_3), \boldsymbol b={}^t (b_1,b_2,b_3) \in \R^3$, 
\begin{equation*}
 \la \boldsymbol a, \boldsymbol b \ra_E =a_1b_1+a_2b_2+a_3b_3, \quad
 \la \boldsymbol a, \boldsymbol b \ra_L =a_1b_1+a_2b_2-a_3b_3,   
\end{equation*}
and 
\begin{equation*}
\boldsymbol a \times_E \boldsymbol b = 
     \begin{pmatrix}
a_2 b_3 - b_2 a_3
\\
a_3 b_1 - b_3 a_1 
\\
a_1 b_2 - b_1 a_2
     \end{pmatrix}, 
\quad
\boldsymbol a \times_L \boldsymbol b = 
     \begin{pmatrix}
a_2 b_3 - b_2 a_3
\\
a_3 b_1 - b_3 a_1 
\\
-(a_1 b_2 - b_1 a_2)
     \end{pmatrix}. 
\end{equation*}
\item 
The Gauss map $\nbf$ is regarded as a vector-valued function, 
indeed, $\nbf \colon M \to S^2 = \{ \xbf \mid \la \xbf, \xbf \ra_E =1 \} \subset \E^3$ 
in the case (i), 
$\nbf \colon M \to H^2 = \{ \xbf \mid \la \xbf, \xbf \ra_L =-1 \}\subset \L^3$ 
in the case (ii), and 
$\nbf \colon M \to S^2_1 = \{ \xbf \mid \la \xbf, \xbf \ra_L =1 \} \subset \L^3$ 
in the case (iii). 
\end{itemize}

\begin{remark}
\begin{enumerate}
 \item 
 As a pre-equal to \eqref{eq:unified-formula}, we have  
\begin{equation}\label{eq:unified-pre-formula}
 -2 H d \xbf = d \nbf + \nbf \times (* d \nbf),   
\end{equation}
so the additional condition in \eqref{eq:co-int-condition} is that 
$\xbf$ does not have zero mean curvature.

Note that 
 the formula \eqref{eq:unified-pre-formula} for the Euclidean case (i) was 
also known in \cite{BFLPP} via a quaternionic description of surfaces. 

\item During the preparing this paper, the author found a paper \cite{SSP}
which discusses a unification of representation formulas from the 
different viewpoint from ours.  
\end{enumerate}
\end{remark}

As a consequence of Theorem \ref{thm:ken-theorem}, 
for surfaces with non-zero constant mean curvature $H$ (called {\it cmc-$H$} for short), 
we can assert the following:
\begin{corollary}\label{cor:cmc-case}
Let $M$ be a connected, oriented $2$-dimensional manifold, 
and suppose that $\xbf \colon M \to N$ is either 
\begin{enumerate}[(i)]
 \item a cmc-$H$ immersion to $N=\E^3$, 
 \item a spacelike cmc-$H$ immersion to $N=\L^3$, or
 \item a timelike cmc-$H$ immersion to $N=\L^3$ 
\end{enumerate}
 with the Gauss map $\nbf$. 
Then $\xbf$ is represented as  
\begin{equation}\label{eq:cmc-H-immersion}
 \xbf = -\frac{1}{2H} \left\{ \nbf + \int \nbf \times (*d \nbf) \right\}, 
\end{equation}  
where $\times$ denotes the Euclidean (resp. Lorentzian) vector product for (i) 
(resp. for (ii), (iii)), and $*$ denotes the Riemannian (resp. Lorentzian) 
Hodge $*$-operator on $T^*M$ with respect to the induced metric
for (i), (ii) (resp. for (iii)).
 
Conversely, suppose that $\nbf$ is a unit vector-valued function 
 on a simply-connected Riemannian or Lorentzian $2$-manifold $M$ satisfying 
\begin{equation}\label{eq:harmonic}
  d*d \nbf \text{ is parallel to } \nbf, \quad 
(d \nbf + \nbf \times * d \nbf \ne 0). 
\end{equation}
Then \eqref{eq:cmc-H-immersion} with a non-zero constant $H$ gives  
\begin{enumerate}[(i)]
 \item a cmc-$H$ surface in $\E^3$ if $M$ is Riemannian and 
$\nbf$ is unit Euclidean vector-valued, 
 \item a spacelike cmc-$H$ surface in $\L^3$ if $M$ is Riemannian and 
$\nbf$ is negative unit Lorentzian vector-valued, or
 \item a timelike cmc-$H$ surface in $\L^3$  if $M$ is Lorentzian and 
$\nbf$ is positive unit Lorentzian vector-valued, 
\end{enumerate}
whose Gauss map and mean curvature are $\nbf$ and $H$.  
\end{corollary}
It is a well known fact that 
 the condition \eqref{eq:harmonic} is equivalent to the 
harmonicity of the map $\nbf \colon M \to S^2$, $H^2$ or $S^2_1$ 
in each case (i), (ii) or (iii), 
and it is also well known that a surface has constant mean curvature if 
and only if its Gauss map is harmonic (cf. \cite{RV}, \cite{KM}).    

We call the formula \eqref{eq:unified-formula} (and its special case 
\eqref{eq:cmc-H-immersion}) 
the {\it Kenmotsu-type formula}.

\begin{remark}
\begin{enumerate}
 \item 
Given a cmc-$H$ surface whose Gauss map is $\nbf$, the map $-\nbf$ is also 
harmonic. By the formula \eqref{eq:cmc-H-immersion} with $-\nbf$, 
we obtain another cmc-$H$ surface 
which is known as the parallel surface of $\xbf$ with constant mean curvature.  
\item
Another well-known fact is that a parallel surface 
$\check \xbf:=\xbf + \frac{1}{2H}\nbf$ 
of a cmc-$H$ surface $\xbf$, i.e.,   
\begin{equation}\label{eq:lelieuvre-type}
 \check \xbf 
= -\frac{1}{2H} \left\{ \int \nbf \times (*d \nbf) \right\}
\end{equation}
has constant Gaussian curvature 
\begin{equation*}
 \begin{cases}
\phantom{-}4 H^2 & \text{in the case (i) or (iii)}, \\
-4 H^2 & \text{in the case (ii)}.   
 \end{cases}
\end{equation*}
We call $\check \xbf$ the {\it cgc-companion} of $\xbf$. 

Incidentally, a representation formula for negative constant Gaussian curvature surface 
in $\E^3$ is known as {\it Lelieuvre's formula}. We thus have obtained a unified 
 {\it Lelieuvre-type} formula \eqref{eq:lelieuvre-type} for  
positive constant Gaussian curvature surfaces in $\E^3$,  for 
spacelike surfaces with negative constant Gaussian curvature in $\L^3$, 
and for timelike surfaces with positive constant Gaussian curvature in $\L^3$. 

We also remark that there are formulas due to Sym \cite{Sym} and Bobenko \cite{Bo}, 
which create surfaces with negative constant Gaussian curvature and 
surfaces with positive constant Gaussian curvature (simultaneously cmc-$H$) 
from harmonic maps to $S^2$, respectively.      
Sym-Bobenko's formula is well-known as one of the most powerful tool nowadays.   
We refer to \cite{MS} for details. 
\end{enumerate} 
\end{remark}

\section{Proof of Kenmotsu-type formula}\label{sec:pf-ken-type}
We shall provide proofs by the moving frame method. 
Notations here are standard (cf. \cite{Br}, etc.). 
\subsection{The case (i)}
Let $e_1, e_2$ be a local orthonormal frame tangent to $\xbf(M) \subset \E^3$. 
We regard $e_1, e_2$ as vector-valued functions. Define a unit normal field $e_3$ 
by 
\begin{equation*}
 e_3 := e_1 \times e_2. 
\end{equation*}
Define local $1$-forms $\omega^i$ ($i=1,2$) and 
$\omega^\alpha_\beta$ ($\alpha, \beta =1,2,3$) by 
\begin{equation}\label{eq:dx-de_a}
 d \xbf = e_i \omega^i, \quad de_\alpha =e_\beta \omega^\beta_\alpha. 
\end{equation}
Note that $(\omega^\alpha_\beta)$ is $\mathfrak{so}(3)$-valued and 
\begin{equation*}
 d \omega^i = -\omega^i_j \wedge \omega^j, \quad 
 d \omega^\alpha_\beta = -\omega_\gamma^\alpha \wedge \omega^\gamma_\beta.  
\end{equation*} 
are satisfied. Moreover, defining $h_{ij}$ by 
\begin{equation*}
 \omega^3_i = h_{ij}\omega^j, 
\end{equation*}
then $h_{ij}=h_{ji}$ and the mean curvature $H$ is, by definition, 
\begin{equation*}
 H =\frac{h_{11}+h_{22}}{2}. 
\end{equation*} 

We shall also use complex-number notation. Set 
\begin{equation*}
 e := \frac{1}{2} (e_1 - \iii e_2), \quad 
 \omega := \omega^1 + \iii \omega^2, \quad 
 \pi := \omega^3_1 - \iii \omega^3_2. 
\end{equation*}
Then we have 
\begin{equation}\label{eq:pi}
 \pi = q \omega + H \bar \omega, \text{ where } 
q= \frac{1}{2}(h_{11}- h_{22} -2 \iii h_{12}). 
\end{equation}
The equation \eqref{eq:dx-de_a} can be rewritten as
\begin{align}\label{eq:dx-de-de_2}
 d \xbf = e \omega + \bar e \bar \omega, \quad 
d e = - \iii e \omega^1_2 + \frac{1}{2} e_3 \pi, \quad 
d e_3 = -e \bar \pi - \bar e \pi. 
\end{align}
Since the Hodge $*$-operator acts as $* \omega = -\iii \omega$, 
it follows from \eqref{eq:pi}, \eqref{eq:dx-de-de_2} that  
\begin{align*}
 * d e_3 = \iii \{ 2H (e\omega -\bar e \bar \omega) 
- e \bar \pi + \bar e \pi \}. 
\end{align*}
Moreover, since $e_3 \times e = \iii e$, we have 
\begin{align*}
 e_3 \times (* d e_3) &= 
-2H ( e\omega + \bar e \bar \omega ) + e \bar \pi + \bar e \pi 
 = -2H d\xbf - d e_3.  
\end{align*}
The map $e_3$ is nothing but the Gauss map $\nbf$, and henceforth 
\begin{equation*}
 d \xbf = -\frac{1}{2H} \{ d \nbf + \nbf \times (*d \nbf) \},  
\end{equation*}
assuming $H \ne 0$. Integrating this, we have the former assertion. 

The latter assertion (i.e., converse assertion) follows from Poincar\'e's lemma. 

\begin{remark}
 The completely integrable condition \eqref{eq:co-int-condition} can be written as 
\begin{equation*}
 \frac{dH}{H} \wedge \{ d \nbf + \nbf \times (*d \nbf) \} = 
\nbf \times (d * d \nbf). 
\end{equation*}
Hence, in the case where $H$ is non-zero constant, 
\eqref{eq:co-int-condition} is equivalent to 
$\nbf \times (d * d \nbf) =0$, i.e., $d * d \nbf$ is parallel to $\nbf$. 
This observation completes the proof of Corollary \ref{cor:cmc-case}. 
\end{remark}

\subsection{The case (ii)} 
Although there are some points to pay attention to, e.g., 
using $e_3 = -e_1 \times_L e_2$, the 1-form $(\omega^\alpha_\beta)$ is 
$\mathfrak{so}(2,1)$-valued, etc, 
the proof is quite similar to the cases (i) above and (iii) below. 
So the proof is omitted here and is left to the reader.   

\subsection{The case (iii)} 
Let $e_1, e_2$ be a local orthonormal frame tangent to the 
timelike surface $\xbf(M) \subset \L^3$ such that 
\begin{equation*}
 \la e_1,e_1 \ra_L =1, \quad \la e_2,e_2 \ra_L = -1, \quad 
\la e_1,e_2 \ra_L = 0. 
\end{equation*} 
We regard $e_1, e_2$ as $\R^3$-valued functions. Define a unit normal field $e_3$ 
by 
\begin{equation*}
 e_3 := -e_1 \times_L e_2. 
\end{equation*}
It should be noted that $e_2 \times_L e_3 = -e_1$ and $e_3 \times_L e_1 = e_2$. 

Define local $1$-forms $\omega^i$ ($i=1,2$) and 
$\omega^\alpha_\beta$ ($\alpha, \beta =1,2,3$) by 
\begin{equation}\label{eq:timelike-dx-de_a}
 d \xbf = e_i \omega^i, \quad de_\alpha =e_\beta \omega^\beta_\alpha. 
\end{equation}
Note that $(\omega^\alpha_\beta)$ is $\mathfrak{so}(1,2)$-valued, i.e., 
$\omega^{\alpha}_{\alpha} =0$, $\omega^1_2 =\omega^2_1$, 
$\omega^1_3 =-\omega^3_1$, $\omega^2_3 =\omega^3_2$. 
The following equations hold:
\begin{equation*}
 d \omega^i = -\omega^i_j \wedge \omega^j, \quad 
 d \omega^\alpha_\beta = -\omega_\gamma^\alpha \wedge \omega^\gamma_\beta.  
\end{equation*} 
Moreover, define $h_{ij}$ by 
\begin{equation*}
 \omega^3_i = h_{ij}\omega^j, 
\end{equation*}
then $h_{ij}=h_{ji}$ and the mean curvature $H$ is, by definition, 
\begin{equation*}
 H =\frac{h_{11}-h_{22}}{2}. 
\end{equation*} 

We shall also use the paracomplex-number notation. 
Recall that  
\begin{equation*}
 \check \C = \{ x + \jjj y \mid x,y \in \R \}, 
\end{equation*}
with the rules of addition and multiplication given by 
\begin{align*}
& (x + \jjj y) + (u + \jjj v) = (x + u) + \jjj ( y+v), \\ 
& (x + \jjj y)(u + \jjj v) = (xu + yv) +\jjj (xv + yu) \qquad
\text{(in particular, $\jjj^2=1$), }
\end{align*}
is a commutative algebra whose elements are called \textit{paracomplex} numbers 
(also called \textit{split complex} numbers).  

Set  
\begin{equation*}
 e := \frac{1}{2} (e_1 + \jjj e_2), \quad 
 \omega := \omega^1 + \jjj \omega^2, \quad 
 \pi := \omega^3_1 + \jjj \omega^3_2. 
\end{equation*}
It is easily verified that 
\begin{equation}\label{eq:timelike-e3times}
 e_3 \times_L e = \jjj e. 
\end{equation}
The Hodge $*$-operator with respect to $I=(\omega^1)^2 -(\omega^2)^2$ 
satisfies, by definition,   
\begin{equation*}
 * \omega^1 = \omega^2, \quad 
 * \omega^2 = \omega^1, 
\end{equation*}
hence
\begin{equation}\label{eq:indefinite-*op}
 * \omega = \jjj \omega. 
\end{equation} 
On the other hand, we have 
\begin{equation}\label{eq:pi-3case}
 \pi = q \omega + H \bar \omega \text{ where } 
q= \frac{1}{2}(h_{11}+ h_{22} +2 \jjj h_{12}). 
\end{equation}
The equation \eqref{eq:timelike-dx-de_a} can be rewritten as
\begin{align}\label{eq:timelike-dx-de-de_2}
 d \xbf = e \omega + \bar e \bar \omega, \quad 
d e_3 = -e \bar \pi - \bar e \pi. 
\end{align}
It follows from \eqref{eq:indefinite-*op}, 
\eqref{eq:pi-3case}, \eqref{eq:timelike-dx-de-de_2} that  
\begin{align*}
 * d e_3 = -\jjj \{ 2H (e\omega -\bar e \bar \omega) 
- e \bar \pi + \bar e \pi \}. 
\end{align*}
Moreover, by \eqref{eq:timelike-e3times}, we have 
\begin{align*}
 e_3 \times_L (* d e_3) &= 
-2H ( e\omega + \bar e \bar \omega ) + e \bar \pi + \bar e \pi 
 = -2H d\xbf - d e_3.  
\end{align*}
The $e_3$ is nothing but the Gauss map $n$, and henceforth 
\begin{equation*}
 d \xbf = -\frac{1}{2H} \{ d \nbf + \nbf \times_L (*d \nbf) \}. 
\end{equation*}
Integrating this, we have the former assertion. 

The latter assertion (i.e., converse assertion) follows from Poincar\'e's lemma.

\section{Lorentzian Delaunay surfaces
}\label{sec:delaunay}
Corollary \ref{cor:cmc-case} asserts that a cmc-$H$ surface 
is determined by 
its intrinsic Riemannian (Lorentzian) structure and a harmonic map to 
$S^2$, $H^2$ or $S^2_1$.  

A \textit{Delaunay surface} is, by definition, a rotational surface with 
non-zero constant mean curvature in $\E^3$, after the work \cite{D}. 
Also, by the terminology (Lorentzian) \textit{Delaunay surface} 
we mean a rotational surface with non-zero constant mean curvature in $\L^3$.  

(Lorentzian) Delaunay surfaces have been studied by many authors.  
We revisit them in relation to Kenmotsu-type formulas. 
Euclidean Delaunay surfaces will be treated in 
Section \ref{sec:cmc-helicoids} as a special case of helicoidal cmc-$H$ surfaces. 
We consider Lorentzian Delaunay surfaces in this section. 
Lorentzian Delaunay surfaces are classified into six types by their causality 
and rotation axes: whether it is a spacelike surface or a timelike surface,  
 whether the rotation axis is timelike, spacelike or lightlike.  
Moreover, the classes of timelike Delaunay surfaces with spacelike axis can be 
divided into two kinds. Therefore, one can say that 
Lorentzian Delaunay surfaces are classified into seven types 
in a slightly finer sense.   

Applying a Lorentzian motion on a Delaunay surface $\xbf$ if necessary,  
we may assume the rotation axis of $\xbf$ is the line through the origin 
in the direction 
\begin{align*}
& {}^t(0,0,1) \text{ if $\xbf$ has timelike rotation axis}, \\
& {}^t(1,0,0) \text{ if $\xbf$ has spacelike rotation axis}, \\
& {}^t(1,0,1) \text{ if $\xbf$ has lightlike rotation axis}. 
\end{align*}
Moreover we can choose parameters $(u,v)$ so that 
$v$ is the rotation parameter and the induced metric is 
conformal either to $du^2 + dv^2$ or to $du^2 - dv^2$. 

According to these normalizations, the Gauss map $\nbf$ of $\xbf$
is written in one of the following forms:  
\begin{enumerate}
 \item 
(for a spacelike rotational surface $\xbf$ with timelike axis)
\begin{equation}\label{gm:srswt} 
\nbf \colon (u,v) \mapsto 
\begin{pmatrix}
 e^{\iii v} & 0 \\ 0 & 1
\end{pmatrix}
\begin{pmatrix}
 \sigma \\ \gamma
\end{pmatrix} 
\in H^2 \subset \C \times \R \cong (\L^3, ++-),  
\end{equation}
 \item 
(for a  timelike rotational surface $\xbf$ with timelike axis)
\begin{equation}\label{gm:trswt}
\nbf \colon (u,v) \mapsto 
\begin{pmatrix}
 e^{\iii v} & 0 \\ 0 & 1
\end{pmatrix}
\begin{pmatrix}
 \gamma \\ \sigma
\end{pmatrix} 
\in S^2_1 \subset \C \times \R \cong (\L^3, ++-), 
\end{equation}
 \item 
(for a spacelike rotational surface $\xbf$ with spacelike axis)
\begin{equation}\label{gm:srsws}
\nbf \colon (u,v) \mapsto 
\begin{pmatrix}
 1 & 0 \\ 0 & e^{-\jjj v}
\end{pmatrix}
\begin{pmatrix}
 \sigma \\ \jjj \gamma 
\end{pmatrix} 
\in H^2 \subset \R \times \check \C \cong (\L^3, ++-),  
\end{equation}
 \item 
 (for a timelike rotational surface $\xbf$ with spacelike axis of the first kind)
\begin{equation}\label{gm:trsws1}
\nbf \colon (u,v) \mapsto 
\begin{pmatrix}
 1 & 0 \\ 0 & e^{-\jjj v}
\end{pmatrix}
\begin{pmatrix}
 \gamma \\ \jjj \sigma 
\end{pmatrix} 
\in S^2_1 \subset \R \times \check \C \cong (\L^3, ++-), 
\end{equation}
\item 
(for a timelike rotational surface $\xbf$ with spacelike axis of the second kind)
\begin{equation}\label{gm:trsws2}
\nbf \colon (u,v) \mapsto
\begin{pmatrix}
 1 & 0 \\ 0 & e^{-\jjj v}
\end{pmatrix}
\begin{pmatrix}
 s \\ c 
\end{pmatrix} 
\in S^2_1 \subset \R \times \check \C \cong (\L^3, ++-),  
\end{equation}
 \item (for a spacelike rotational surface $\xbf$ with lightlike axis)
\begin{equation}\label{gm:srswl}
\nbf \colon (u,v) \mapsto
 \exp vA \cdot  
\left(
\begin{smallmatrix}
\sigma \\ 0 \\ \gamma 
\end{smallmatrix}
\right) \in H^2, 
\end{equation}

\item (for a timelike rotational surface $\xbf$ with lightlike axis)
\begin{equation}\label{gm:trswl}
\nbf \colon (u,v) \mapsto 
 \exp vA \cdot  
\left(
\begin{smallmatrix}
\gamma \\ 0 \\ \sigma 
\end{smallmatrix}
\right) \in S^2_1, 
\end{equation}
\end{enumerate}
where $\sigma = \sigma(u)$, $\gamma = \gamma(u)$ are functions with 
$\gamma^2 - \sigma^2 =1$, and 
$s=s(u)$, $c=c(u)$ are functions with $c^2+s^2=1$, and 
\begin{equation*}
 A = 
\begin{pmatrix}
 0 & -1 & 0 \\ 1 & 0 & -1 \\ 0 & -1 & 0
\end{pmatrix} 
\in 
\mathfrak{so}(2,1). 
\end{equation*}
We remark that one can start \eqref{gm:srsws}--\eqref{gm:trsws2} with $e^{\jjj v}$ 
instead of $e^{-\jjj v}$. However, we prefer $e^{-\jjj v}$ 
for the matching with Remark \ref{rem:gc-property} stated later.  

As mentioned above, we provide 
a conformal structure $[du^2 + dv^2]$ on the domain $(U; u,v)$ 
for \eqref{gm:srswt}, \eqref{gm:srsws}, \eqref{gm:srswl}, and  
a Lorentzian conformal structure $[du^2 - dv^2]$  
for \eqref{gm:trswt}, \eqref{gm:trsws1}, \eqref{gm:trsws2}, \eqref{gm:trswl}. 
It should be noted that each harmonic map equation for 
\eqref{gm:srswt}--\eqref{gm:trsws1} is 
the same one, indeed,  
\begin{equation}\label{eq:hm-eq}
 (\sigma' \gamma - \sigma \gamma')' -\sigma \gamma =0 
\quad (\text{with } \gamma^2-\sigma^2=1). 
\end{equation}
Moreover, the equation \eqref{eq:hm-eq} can be explicitly solved in terms of 
Jacobi's elliptic functions as follows: 
\begin{equation}\label{sol:hm-eq}
 \sigma = \cs (u,k), \quad 
 \gamma = \ns (u, k) \qquad (-\infty < k^2 < \infty), 
\end{equation}
where $u$ can be replaced by $\pm u + C$ for any constant $C$, 
and $-\infty < k^2 < \infty$ means that the modulus $k$ can 
be any real or pure imaginary number. 

Note that the harmonic map equation for \eqref{gm:trsws2} is 
\begin{equation}\label{eq:hm-eq-sc}
(s'c -sc')' + sc =0 \quad (\text{with }c^2+s^2=1),   
\end{equation}
which has the solution 
\begin{equation}\label{sol:hm-eq-wS}
 s = k \sn (u,k), \quad c = \pm \dn (u,k)  \qquad (0 \le k < \infty), 
\end{equation}
where $u$ can be replaced by $\pm u + C$ for any constant $C$. 

Note also that each harmonic map equation for \eqref{gm:srswl} or for  
\eqref{gm:trswl} is
\begin{equation}\label{eq:hm-eq-wL}
 (\sigma' \gamma - \sigma \gamma')' + (\sigma -\gamma)^2 =0 
\quad (\text{with } \gamma^2-\sigma^2=1). 
\end{equation}
With the setting 
\begin{equation*}
 \sigma = \frac{1}{2} \left( \phi - \frac{1}{\phi} \right), \ 
 \gamma = \frac{1}{2} \left( \phi + \frac{1}{\phi} \right),   
\end{equation*}
the equation \eqref{eq:hm-eq-wL} is equivalent to 
\begin{equation}\label{eq:hm-eq-wL'}
 \phi \phi'' - (\phi')^2 +1 =0, 
\end{equation}
which has solutions 
\begin{equation}\label{sol:hm-eq-wL}
\phi =u \quad \text{or} \quad 
\phi = \frac{\sin ku}{k},  \qquad (-\infty < k^2 < \infty, k \ne 0), 
\end{equation}
where $u$ can be replaced by $\pm u + C$ for any constant $C$.



Without loss of generality, the value of mean curvature $H$ is assumed to be $-1/2$.  
As an application of \eqref{gm:srswt}--\eqref{sol:hm-eq-wL} and 
Corollary \ref{cor:cmc-case}, we have simple and explicit expressions of 
Lorentzian Delaunay surfaces as follows: 
\begin{itemize}
 \item Delaunay surfaces with timelike axis
\begin{itemize}
 \item spacelike ones
\begin{equation}\label{eq:sDSwT}
\xbf(u,v) =
\begin{pmatrix}
 e^{\iii v} & 0 \\ 0 & 1
\end{pmatrix}
\left\{  
\begin{pmatrix}
 \sigma \\ \gamma
\end{pmatrix}
+ 
\begin{pmatrix}
 \sigma'/\gamma \\
\int \sigma^2
\end{pmatrix}
\right\} \in \C \times \R \cong \L^3
\end{equation}
 \item timelike ones
\begin{equation}\label{eq:tDSwT}
\xbf(u,v)=
\begin{pmatrix}
 e^{\iii v} & 0 \\ 0 & 1
\end{pmatrix}
\left\{  
\begin{pmatrix}
 \gamma \\ \sigma
\end{pmatrix}
+ 
\begin{pmatrix}
 -\sigma'/\gamma \\
-\int \gamma^2
\end{pmatrix}
\right\} \in \C \times \R \cong \L^3, 
\end{equation} 
\end{itemize}
 \item Delaunay surfaces with spacelike axis
\begin{itemize}
 \item spacelike ones
\begin{equation}\label{eq:sDSwS}
\xbf(u,v)=
\begin{pmatrix}
 1 & 0 \\ 0 & e^{- \jjj v} 
\end{pmatrix}
\left\{
\begin{pmatrix}
 \sigma \\ \jjj \gamma
\end{pmatrix}
+ 
\begin{pmatrix}
- \int \gamma^2 \\ - \jjj \sigma'/\gamma
\end{pmatrix}
\right\}\in \R \times \check \C \cong \L^3
\end{equation}
 \item timelike ones of the first kind
\begin{equation}\label{eq:tDSwS1}
\xbf(u,v)=
\begin{pmatrix}
 1 & 0 \\ 0 & e^{- \jjj v} 
\end{pmatrix}
\left\{  
\begin{pmatrix}
 \gamma \\ \jjj \sigma
\end{pmatrix}
+ 
\begin{pmatrix}
\int \sigma^2  \\ \jjj \sigma'/\gamma 
\end{pmatrix}
\right\} \in \R \times \check \C \cong \L^3
\end{equation}
 \item timelike ones of the second kind
\begin{equation}\label{eq:tDSwS2}
\xbf(u,v)=
\begin{pmatrix}
 1 & 0 \\ 0 & e^{-\jjj v} 
\end{pmatrix}
\left\{  
\begin{pmatrix}
 s \\ c 
\end{pmatrix}
+ 
\begin{pmatrix}
- \int c^2 \\ - s'/ c  
\end{pmatrix}
\right\} \in \R \times \check \C \cong \L^3, 
\end{equation}
\end{itemize}
 \item Delaunay surfaces with lightlike axis
\begin{itemize}
 \item spacelike ones
\begin{equation}\label{eq:sDSwL}
\xbf(u,v) = 
\exp vA \cdot 
\left\{
\left(
\begin{matrix}
  \frac{1}{2} (\phi - \frac{1}{\phi}) \\ 0 \\  \frac{1}{2} (\phi + \frac{1}{\phi})
\end{matrix}\right)
- \int 
\left(
\begin{matrix}
 \frac{1}{2} ( 1+ \frac{1}{\phi^2}) \\ 0 \\  \frac{1}{2} (1-\frac{1}{\phi^2})
\end{matrix}\right) 
\right\}
\end{equation}
 \item timelike ones
\begin{equation}\label{eq:tDSwL}
\xbf(u,v) = 
\exp vA \cdot   
\left\{
\left(\begin{matrix}
 \frac{1}{2} (\phi + \frac{1}{\phi}) \\ 0 \\ \frac{1}{2} (\phi - \frac{1}{\phi})
\end{matrix}\right)
- \int 
\left(\begin{matrix}
\frac{1}{2}( 1- \frac{1}{\phi^2}) \\ 0 \\ \frac{1}{2} (1+\frac{1}{\phi^2})
\end{matrix}\right) 
\right\}, 
\end{equation}
\end{itemize}
\end{itemize}
where $(\sigma,\gamma)=(\sigma(u),\gamma(u))$, 
$(s,c)=(s(u),c(u))$ and $\phi = \phi(u)$ are solutions to 
\eqref{eq:hm-eq}, \eqref{eq:hm-eq-sc} and \eqref{eq:hm-eq-wL'}, respectively, 
i.e., functions given in \eqref{sol:hm-eq}, \eqref{sol:hm-eq-wS} and 
\eqref{sol:hm-eq-wL}. 
\begin{remark}
\begin{enumerate}
 \item In \eqref{eq:sDSwT}--\eqref{eq:tDSwS2}, 
each term including an integral can be also expressed by Jacobi's elliptic functions 
and the elliptic integral $E$ of the second kind: 
 \begin{align*}
& \int \sigma^2 du = 
\begin{cases}
- \left\{ \frac{ \cn \dn}{\sn} + E \circ \am \right\} (u,k) & 
\text{(if $0 \le k \le1$)} \\
(k^2-1)u -k 
\left\{ \frac{ \cn \dn}{\sn} + E \circ \am \right\} (ku, 1/k)
& \text{(if $k>1$)} \\
-\alpha \left\{ \frac{ \cn \dn}{\sn} + E \circ \am \right\} (\alpha u, \beta) 
& \text{(if $k \in \iii \R$)} , 
\end{cases} 
\intertext{where $\alpha=\sqrt{1-k^2}$, $\beta=\sqrt{-k^2/(1-k^2)}$, } 
& \int \gamma^2 du = u + \int \sigma^2 du,  \\
& \int c^2 du = 
\begin{cases}
 E \circ \am (u,k) & \text{if $0 \le k <1$} \\
(1-k^2)u + k E \circ \am (ku, 1/k) & \text{if $k > 1$.}
\end{cases}
\end{align*}

On the other hand, 
all terms in \eqref{eq:sDSwL}, \eqref{eq:tDSwL} are elementary functions.

 \item 
Some of the Lorentzian Delaunay surfaces have conical singularities,  
namely,
it can happen that $\xbf$ is not a regular surface.  
\end{enumerate}
\end{remark}
\begin{remark}\label{rem:gc-property}
\begin{enumerate}
 \item
Observing \eqref{eq:sDSwT} and \eqref{eq:tDSwS1}, one notices that 
the generating curves of a spacelike Delaunay surface with timelike axis and 
a timelike Delaunay surface with spacelike axis of the first kind  
are coincident, after the reflection $R : x \leftrightarrow z$ 
in the $xz$-plane with respect to the line $x=z$. Note that $R \not\in O(2,1)$. 
In other words, if a curve 
$\Gamma(u)=(\sigma + \sigma'/\gamma, 0 , \gamma + \int \sigma^2 du)$ is given in 
the $xyz$-space, then the trajectory of $\Gamma(u)$ by 
$\begin{pmatrix}
e^{\iii v} & 0 \\ 0 & 1
\end{pmatrix}$ 
is a spacelike Delaunay surface with timelike axis in $\L^3$, 
whereas  
the trajectory of $\Gamma(u)$ by 
$\begin{pmatrix}
1 & 0 \\ 0 & e^{-\jjj v}
\end{pmatrix}
\circ R$ 
is a timelike Delaunay surface with spacelike axis in $\L^3$. 
The same phenomenon holds for cgc-companions $\check \xbf$ 
of \eqref{eq:sDSwT} and \eqref{eq:tDSwS1}. 

We note that the situation stated above also occurs for the pair 
\eqref{eq:tDSwT} and \eqref{eq:sDSwS}.     

 \item 
The profile curve of a timelike Delaunay surface with spacelike axis 
of the second kind
is the same as for a Delaunay surface in $\E^3$ 
(ignoring the underlying Riemannian or Lorentzian structure). 
See Figures \ref{fig:gc-d-1}--\ref{fig:gc-d-4}.  
\end{enumerate}
\end{remark}

\begin{figure}[htbp]
\begin{tabular}{cc}
\begin{minipage}{0.45\hsize}
\centering
 \includegraphics[width=\textwidth,clip]{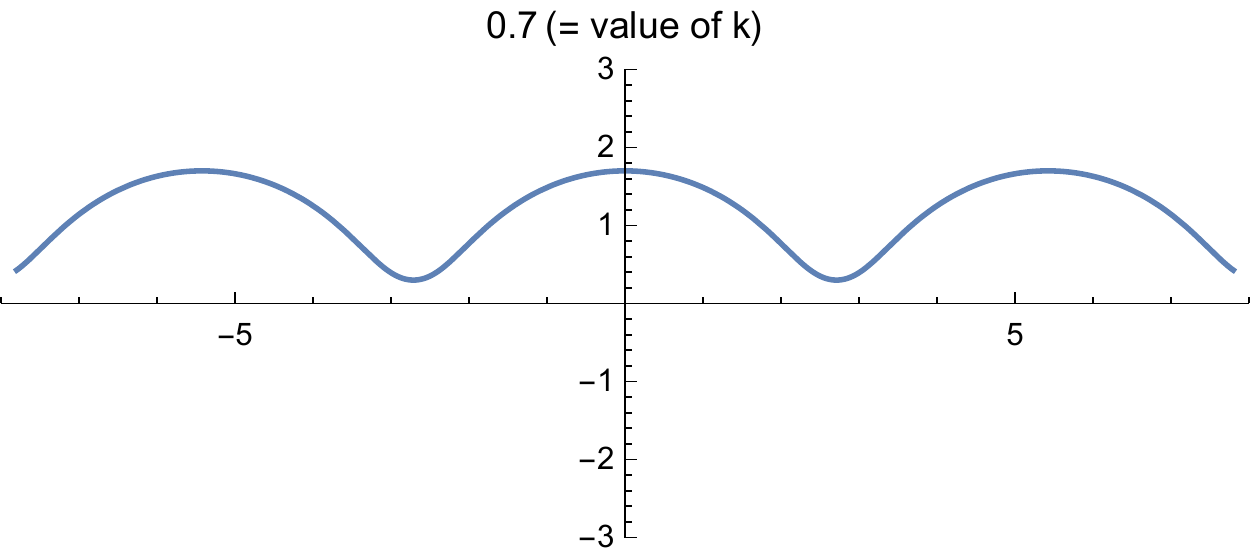}  
\caption{a profile curve ($0<k<1$)}
\label{fig:gc-d-1}
\end{minipage}
\quad
\begin{minipage}{0.45\hsize}
\centering
 \includegraphics[width=\textwidth,clip]{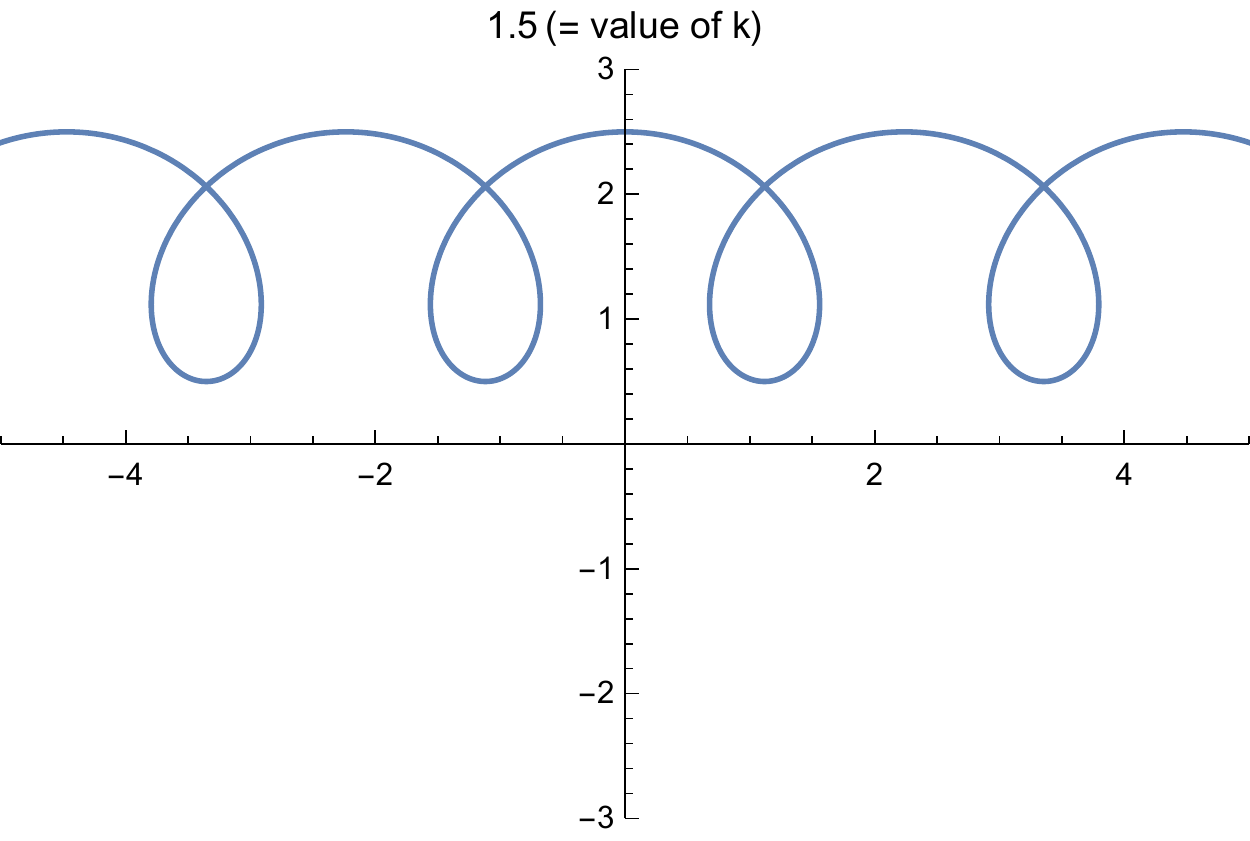} 
\caption{a profile curve ($k>1$)}
\label{fig:gc-d-2}
\end{minipage}
\end{tabular}
\end{figure}

\begin{figure}[htbp]
\begin{tabular}{cc}
\begin{minipage}{0.45\hsize}
\centering
 \includegraphics[width=\textwidth,clip]{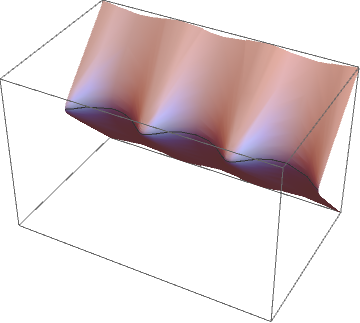}  
\caption{a timelike Delaunay surface with spacelike axis ($0<k<1$)}
\label{fig:gc-d-3}
\end{minipage}
\quad
\begin{minipage}{0.45\hsize}
\centering
 \includegraphics[width=\textwidth,clip]{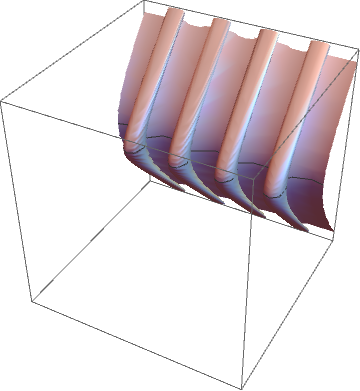} 
\caption{a timelike Delaunay surface with spacelike axis ($k>1$)}
\label{fig:gc-d-4}
\end{minipage}
\end{tabular}
\end{figure}

For previous works on details for Lorentzian Delaunay surfaces, 
we refer to \cite{HN}, \cite{H}, \cite{IH}, \cite{LV1}, \cite{LV2}, \cite{L0}, \cite{L} and so on.

\section{cmc helicoids}\label{sec:cmc-helicoids}
We call a helicoidal surface with non-zero constant mean curvature $H$ in $\E^3$
a {\it cmc-$H$ helicoid}. Cmc-$H$ helicoids were studied and 
already classified  by do Carmo-Dajczer \cite{doCD}. However, in this section, 
we investigate them again by another approach 
using the Kenmotsu formula \eqref{eq:cmc-H-immersion}. 

\subsection{Formulation}
We consider helicoidal motions in $\E^3$. (A helicoidal motion is also called a 
screw motion.) 
Without loss of generality, 
we always suppose the axis of any helicoidal motion is the $z$-axis. Thus 
a group of \textit{helicoidal} motions of \textit{pitch} $\lambda$ is   
a one-parameter group of Euclidean motions 
\begin{equation*}
M_{\lambda}(t) \colon 
\begin{pmatrix}
 x \\ y \\ z
\end{pmatrix}
\mapsto 
 \begin{pmatrix}
  \cos t & - \sin t & 0 \\ \sin t & \cos t & 0 \\ 0 & 0 & 1
 \end{pmatrix}
\begin{pmatrix}
 x \\ y \\ z
\end{pmatrix}
+
\begin{pmatrix}
 0 \\ 0 \\ \lambda t
\end{pmatrix}
\end{equation*}
for a real constant $\lambda$. 
Note that a helicoidal motion of pitch $0$ is a rotational motion. 

A trajectory of a space curve $\cbf \colon I \to \E^3$ by a group of helicoidal motions 
of pitch $\lambda$ is called 
a \textit{helicoidal surface of pitch $\lambda$ with generating curve} $\cbf$.  
A generating curve is not unique for a helicoidal surface. 
We can 
choose a generating curve $\cbf(u)={}^t(x(u),y(u),z(u))$  
so that $(u,v) \mapsto M_{\lambda}(v) \cbf(u)$ is an orthogonal parametrization.  
Then, the first fundamental form is given by 
\begin{equation*}
 \fff =\{ (x')^2+(y')^2 +(z')^2 \}du^2 + \{\lambda^2 + x^2 + y^2 \} dv^2 . 
\end{equation*}
Moreover, after a suitable parameter change of $u$, we have an  
isothermal parametrization $(u,v) \mapsto M_{\lambda}(v)\cbf(u)$, that is, 
$\fff$ is proportional to $du^2 + dv^2$. We suppose this parametrization 
on any helicoidal surface $\xbf = M_{\lambda}(v) \cbf(u)$. 
The Gauss map is defined to be $\nbf = (\xbf_u \times \xbf_v)/|\xbf_u|$.  
It has the form 
\begin{equation}\label{eq:gmap-of-heli}
 \nbf = 
 \begin{pmatrix}
  \cos v & - \sin v & 0 \\ \sin v & \cos v & 0 \\ 0 & 0 & 1
 \end{pmatrix}
\begin{pmatrix}
 \alpha(u) \\ \beta(u) \\ \gamma(u)
\end{pmatrix} , \quad 
\alpha^2 + \beta^2 + \gamma^2 =1. 
\end{equation}

For simplicity, we frequently identify $\E^3$ with $\C \times \R$ 
similarly to the identification in Section \ref{sec:delaunay}. 
For example, \eqref{eq:gmap-of-heli} is also expressed as    
\begin{equation}\label{eq:gmap-of-heli-cpx}
\nbf = 
\begin{pmatrix}
 e^{\iii v} & 0 \\ 0 & 1
\end{pmatrix}
 \begin{pmatrix}
  X \\ \gamma
 \end{pmatrix}\in \C \times \R = \E^3 , \quad 
|X|^2 + \gamma^2 =1, \ (X= \alpha + \iii \beta). 
\end{equation}
The following formulas will often be needed. 
\begin{lemma}
For  $\boldsymbol x = \left(\begin{smallmatrix}
	 X \\ x
	\end{smallmatrix}\right), 
\ybf = \left(\begin{smallmatrix}
	 Y \\ y
	\end{smallmatrix}\right)
\in \C \times \R \cong \E^3$,
\begin{equation*}
 \la \boldsymbol x , \ybf \ra = \Re(X \bar Y) + xy, \quad 
\boldsymbol x \times \ybf = 
\begin{pmatrix}
 -\iii 
\left|
\begin{smallmatrix}
 X & Y \\ x & y
\end{smallmatrix}
\right| \\ 
\Im (\bar X Y)
\end{pmatrix}.
\end{equation*}
\end{lemma}

First of all, we find the harmonic map equation for the map $\nbf$ given by 
\eqref{eq:gmap-of-heli-cpx}. 
Paying attention to the isothermality of $(u,v)$, we have
\begin{align*}
 d \nbf &= 
\begin{pmatrix}
 e^{\iii v} & 0 \\ 0 & 1
\end{pmatrix}
\left\{
 \begin{pmatrix}
  X' \\ \gamma'
 \end{pmatrix} du 
+
 \begin{pmatrix}
  \iii X \\ 0
 \end{pmatrix} dv 
\right\}, \\
* d \nbf &= 
\begin{pmatrix}
 e^{\iii v} & 0 \\ 0 & 1
\end{pmatrix}
\left\{
- \begin{pmatrix}
  \iii X \\ 0
 \end{pmatrix} du 
+
 \begin{pmatrix}
  X' \\ \gamma'
 \end{pmatrix} dv
\right\} , \\ 
 \nbf \times (* d \nbf) &= 
\begin{pmatrix}
 e^{\iii v} & 0 \\ 0 & 1
\end{pmatrix}
\left\{
\begin{pmatrix}
 X \gamma \\ -|X|^2
\end{pmatrix}du 
+
\begin{pmatrix}
 \iii \{ X' \gamma - X \gamma' \}  \\ 
 \Im (X' \bar X)
\end{pmatrix}dv
\right\}. 
\end{align*}
It follows that
\begin{equation*}
 d \nbf + \nbf \times (* d \nbf) = 0 \iff 
 \begin{cases}
  X' + X \gamma =0 \\ 
\gamma' -1 + \gamma^2 =0 \\ 
X + X' \gamma - X \gamma' =0 \\ 
 \Im (X' \bar X) =0
 \end{cases}
\iff
 \begin{cases}
  X' + X \gamma =0 \\ 
\gamma' = 1 - \gamma^2  
 \end{cases}. 
\end{equation*}
We have noted that the equation $d \nbf + \nbf \times (* d \nbf) = 0$ is 
a condition that $H=0$, in other words, 
$\nbf$ is an orientation-reversing conformal map.    

On the other hand, 
\begin{align*}
d * d \nbf &=
\begin{pmatrix}
 e^{\iii v} & 0 \\ 0 & 1
\end{pmatrix}
\left\{
 \begin{pmatrix}
  X'' \\ \gamma''
 \end{pmatrix}
-
 \begin{pmatrix}
  X \\ 0
 \end{pmatrix}
\right\} du \wedge dv,   
\end{align*}
and hence $d * d \nbf$ is proportional to $\nbf$ if and only if 
\begin{align*}
 \begin{vmatrix}
  X'' -X & X \\ \gamma'' & \gamma
 \end{vmatrix}=0, \quad 
\Im ( \overline{X''-X} )X =0, 
\end{align*}
equivalently, 
\begin{equation*}
 (X' \gamma -X \gamma')' = X \gamma, \quad \Im (X' \bar X) =c_1(=\text{constant}). 
\end{equation*}
We have thus the following lemma.
\begin{lemma}\label{lem:de-gm-heli}
 Let a map $\nbf \colon (U, [du^2+dv^2]) \to S^2 \subset \C \times \R \cong \E^3$ 
be given in the form  
\begin{equation}\label{eq:n:UtoS2}
 \nbf = 
\begin{pmatrix}
 e^{\iii v} & 0 \\ 0 & 1
\end{pmatrix}
\begin{pmatrix}
 X(u) \\ \gamma(u) 
\end{pmatrix}, 
\quad |X|^2 +\gamma^2 =1. 
\end{equation}
Then the map $\nbf$ is 
\begin{itemize}
 \item orientation-reversing and conformal if and only if 
\begin{equation}\label{conf-eq-toS2}
  X' + X \gamma =0, \quad
\gamma' = 1 - \gamma^2,   
\end{equation}
\item
harmonic if and only if
\begin{equation}\label{hm-eq-toS2}
 (X' \gamma -X \gamma')' = X \gamma, \quad \Im (X' \bar X) =c_1(=\text{constant}). 
\end{equation}
\end{itemize}
\end{lemma}

Lemma \ref{lem:de-gm-heli} leads to the following proposition.
\begin{proposition}\label{prop:pre-cmc-helicoid}
Let 
\begin{equation*}
 \nbf = 
\begin{pmatrix}
 e^{\iii v} & 0 \\ 0 & 1
\end{pmatrix}
\begin{pmatrix}
 X(u) \\ \gamma(u)
\end{pmatrix}
: (U, [du^2+dv^2]) \to S^2 \subset \C \times \R =\E^3
\end{equation*}
be a harmonic map satisfying $d \nbf + \nbf \times * d\nbf \ne 0$. 
The cmc-$H$ surface whose Gauss map is $\nbf$ is given by 
\begin{equation*}
-2 H \xbf = 
\begin{pmatrix}
 e^{\iii v} & 0 \\ 0 & 1
\end{pmatrix}
\left\{ 
\begin{pmatrix}
 X \\ \gamma
\end{pmatrix}
+
\begin{pmatrix}
 X' \gamma -X \gamma' \\ \int (\gamma^2 -1)du
\end{pmatrix}
\right\} 
+ 
\begin{pmatrix}
 0 \\ c_1 v
\end{pmatrix}, 
\end{equation*}
where $c_1 = \Im (\bar X X') (=\text{constant})$. 

The cgc-companion $\check \xbf$ of $\xbf$ is 
\begin{equation*}
\begin{pmatrix}
 e^{\iii v} & 0 \\ 0 & 1
\end{pmatrix}
\left\{ 
\begin{pmatrix}
 X' \gamma -X \gamma' \\ \int (\gamma^2 -1)du
\end{pmatrix}
\right\} 
+ 
\begin{pmatrix}
 0 \\ c_1 v
\end{pmatrix} . 
\end{equation*}
\end{proposition}
\begin{proof}
 By \eqref{hm-eq-toS2}, we have 
\begin{align*}
 \nbf \times (* d \nbf) &= 
\begin{pmatrix}
 e^{\iii v} & 0 \\ 0 & 1
\end{pmatrix}
\left\{
\begin{pmatrix}
 (X' \gamma - X \gamma')' \\  \gamma^2-1
\end{pmatrix}du 
+
\begin{pmatrix}
 \iii (X' \gamma - X \gamma') \\ 
 c_1
\end{pmatrix} dv
\right\} \\ 
&=
\begin{pmatrix}
d \{ e^{\iii v} (X' \gamma - X \gamma') \} \\ (\gamma^2-1)du +c_1 dv
\end{pmatrix}. 
\end{align*}
Thus 
\begin{align*}
 \int \nbf \times (* d \nbf) &= 
\begin{pmatrix}
 e^{\iii v} (X' \gamma - X \gamma') \\ \int (\gamma^2-1)du +c_1 v
\end{pmatrix}
=
\begin{pmatrix}
 e^{\iii v} & 0 \\ 0 & 1
\end{pmatrix}
\begin{pmatrix}
 X' \gamma - X \gamma' \\ \int (\gamma^2 -1)du
\end{pmatrix}
+ 
\begin{pmatrix}
 0 \\ c_1 v
\end{pmatrix}.
\end{align*}
Hence the assertion follows from Corollary \ref{cor:cmc-case}. 
\end{proof}

The cmc-$H$ surface obtained in Proposition \ref{prop:pre-cmc-helicoid} is, 
of course, a cmc-$H$ helicoid. 

\subsection{Parametrization by the explicit solution}
We wish to find an explicit description of a  
harmonic map $\nbf$ which is not orientation-reversing conformal. 
For this purpose, we first determine the solution to the 
orientation-reversing conformal map equation \eqref{conf-eq-toS2}. 
\begin{lemma}
Let a map $\nbf \colon (U, [du^2+dv^2]) \to S^2$ 
be given in the form \eqref{eq:n:UtoS2}. Then the map $\nbf$ 
solves the equation \eqref{conf-eq-toS2}  
if and only if 
\begin{equation}\label{eq:conf-map-sol}
 \nbf = 
\begin{pmatrix}
 0 \\ \pm 1
\end{pmatrix}
\text{ or }
\begin{pmatrix}
 e^{\iii v} & 0 \\ 0 & 1
\end{pmatrix}
\begin{pmatrix}
1/ \cosh u \\ \tanh u 
\end{pmatrix}
\end{equation}  
up to a translation $(u,v) \mapsto (u+c_1, v+c_2)$ in the $uv$-plane. 
\end{lemma}
\begin{proof}
The assertion is immediately verified  and is left to the reader.  
\end{proof}

Other than \eqref{eq:conf-map-sol}, 
we seek the solution to \eqref{hm-eq-toS2}. 
Since $|X|^2+\gamma^2=1$, we set 
\begin{align*}
\begin{pmatrix}
 X \\ \gamma
\end{pmatrix}
&= 
\begin{pmatrix}
 e^{\iii g} \cos f \\ \sin f
\end{pmatrix}
\end{align*}
for $f=f(u)$, $g=g(u)$. Then we have
\begin{align}
 \eqref{hm-eq-toS2} & \iff 
\begin{cases}
  f'' + (1+(g')^2)\cos f \sin f  =0 \\
 f' g' -(g' \cos f \sin f )' =0  \\ 
 g' \cos^2 f = c_1
\end{cases} \notag \\
& \iff
 \begin{cases}
 f'' + \{ 1+(g')^2 \} \cos f \sin f  =0 \\ 
 g' \cos^2 f = c_1  
 \end{cases}. \label{eq:hm-eq-reduced} 
\end{align}
If a solution to \eqref{eq:hm-eq-reduced} consists of 
a constant function $f$ other than 
\eqref{eq:conf-map-sol}, then $\sin f =0$ and $g=c_1 u (+C)$ $(c_1 \ne 0)$. 
In other words, the solution to \eqref{hm-eq-toS2} which has constant $\gamma$ is 
\begin{equation}\label{eq:special-sol}
 \begin{pmatrix}
  X \\ \gamma
 \end{pmatrix}
=
\begin{pmatrix}
 e^{\iii c_1 u} \\ 0 
\end{pmatrix}, \quad 
c_1 \ne 0. 
\end{equation}
Applying Proposition \ref{prop:pre-cmc-helicoid} with \eqref{eq:special-sol}, 
we have 
\begin{equation}\label{eq:c-cylinder-as-heli}
 \begin{aligned}
 -2H \xbf &= 
\begin{pmatrix}
 e^{iv} & 0 \\ 0 & 1
\end{pmatrix}
\left\{
\begin{pmatrix}
 \pm e^{i c_1 u} \\ 0 
\end{pmatrix}
+
\begin{pmatrix}
 0 \\ \int (-1) du
\end{pmatrix}
\right\}
+ 
\begin{pmatrix}
 0 \\ c_1 v
\end{pmatrix} \\
&=
\begin{pmatrix}
 e^{i(v \pm c_1 u)} \\ -u + c_1 v
\end{pmatrix}. 
\end{aligned}
\end{equation}
Note that \eqref{eq:c-cylinder-as-heli} represents a circular cylinder.

Next, we consider the system \eqref{eq:hm-eq-reduced} of equations  
under the assumption that $f$ is non-constant. 
Substituting the second equation into the first, we have 
\begin{equation*}
 f'' + c_1^2 \frac{\sin f}{\cos^3 f} + \sin f \cos f=0. 
\end{equation*}
Multiplying by $f'$ on both sides and integrating them, we have 
\begin{equation*}
 (f')^2 + \frac{c_1^2}{\cos^2 f} -\cos^2 f = c_2+1 
\end{equation*}
for some constant $c_2$. 
(On the right hand side, we have used $c_2 +1$ instead of just $c_2$ 
for later convenience.)
\begin{equation*}
 \{ \cos f \cdot f'\}^2 +c_1^2 -\cos^4 f = (c_2+1) \cos^2 f
\end{equation*}
Recalling that $\sin f = \gamma$, we have
\begin{align}
 (\gamma')^2 &=(1-\gamma^2)^2 +(c_2+1) (1-\gamma^2) -c_1^2 \notag \\
&= \gamma^4 -(c_2+1) \gamma^2 +(c_2-c_1^2). \label{al:de-4-gamma}
\end{align}
So the solution $\gamma$ must satisfy the ordinary differential equation 
\eqref{al:de-4-gamma}. To analyze \eqref{al:de-4-gamma}, we may assume 
$c_1 \ge 0$. At the same time, we have to take it into consideration that 
the range of $\gamma$ must $-1 \le \gamma \le 1$. Hence, the constants 
$c_1$, $c_2$ receive restrictions on their values. 
Indeed, we have the following lemma. 
\begin{lemma}
 Let ${}^t (X, \gamma)$ be a solution to \eqref{hm-eq-toS2} 
with $|X|^2 + \gamma^2=1$ which is neither \eqref{eq:conf-map-sol} 
nor \eqref{eq:special-sol}. Then $\gamma$ satisfies the differential 
equation \eqref{al:de-4-gamma} for constants $c_1, c_2$ with 
\begin{equation*}
 c_2 > c_1^2. 
\end{equation*}
\end{lemma}
\begin{proof}
 We have only to show that $c_2 > c_1^2$. 

For a solution $\gamma = \gamma(u)$, 
there exists $u=u_0$ such that 
\begin{equation*}
(\gamma')^2 >0, \text{ i.e., } 
\gamma^4 -(c_2+1) \gamma^2 +(c_2-c_1^2)>0,  
\end{equation*}
and thus the following assertion necessarily holds: 
there exists $x \in [0,1]$ such that  
\begin{equation*}
F(x) = x^2 -(c_2+1) x +(c_2-c_1^2)>0.  
\end{equation*}
Since $F(0)= c_2 -c_1^2$, $F(1) = -c_1^2 \le 0$ and $F(x)$ is convex below, 
$F(0)= c_2 -c_1^2$ must be positive.  
\end{proof}
We continue to investigate \eqref{al:de-4-gamma} with $c_2 > c_1^2$.
Let 
\begin{equation*}
 \mathcal C := \{ (c_1,c_2) \in \R^2 \mid c_1 \ge 0, \ c_2 > c_1^2 \}. 
\end{equation*}
For $(c_1,c_2) \in \mathcal C$, the roots of the quadric equation 
$x^2-(c_2+1)x +(c_2-c_1^2)=0$ are 
\begin{equation*}
\alpha = \frac{c_2+1 - \sqrt{(c_2-1)^2 + 4c_1^2}}{2}, \ 
\beta =  \frac{c_2+1 + \sqrt{(c_2-1)^2 + 4c_1^2}}{2}. 
\end{equation*} 
By these, we have a bijective correspondence 
\begin{equation*}
\mathcal C = \{ c_1 \ge 0, c_2 > c_1^2 \}  
\longleftrightarrow 
\mathcal A' =  \{ (\alpha, \beta) \mid 0 < \alpha \le 1, 1 \le \beta < \infty \}. 
\end{equation*}
The inverse is 
\begin{equation*}
 \begin{cases}
  c_1 = \sqrt{(1-\alpha)(\beta-1)} \\ 
  c_2 = \alpha + \beta -1. 
 \end{cases}
\end{equation*}
Moreover, we set 
\begin{equation*}
 a := \sqrt{\alpha}
, \quad 
 b :=\sqrt{\beta}
\end{equation*}
and have a bijective correspondence
\begin{equation*}
\mathcal C \longleftrightarrow 
\mathcal A := \{ (a,b) \mid a \in (0,1], b \in [1,\infty) \}. 
\end{equation*}
Therefore, the differential equation \eqref{al:de-4-gamma} 
with $(c_1, c_2) \in \mathcal C$
is rewritten as 
\begin{equation*}
 (\gamma')^2 = (a^2 -\gamma^2)(b^2 -\gamma^2), \ (a,b) \in \mathcal A. 
\end{equation*}
The general solution is 
\begin{equation*}
 \gamma = a \sn (bu , \frac{a}{b}) \left( = b \sn (au , \frac{b}{a}) \right), 
\end{equation*}
where $u$ can be replaced by $\pm u +C$ for an arbitrary constant $C$.  
Note that the range of $\gamma = a \sn (bu , \frac{a}{b})$ is included in 
$[-1,1]$ for each fixed $(a,b) \in \mathcal A$. 

Without loss of generality, we may assume $C=0$. 
Thus we think of $\gamma = \pm a \sn (bu , \frac{a}{b})$. 
It will be explained later that the choice of the sign of 
$\gamma = \pm a \sn (bu , \frac{a}{b})$ does not create any essential difference. 
So we concentrate on $\gamma = - a \sn (bu , \frac{a}{b})$ for a while. 
%
%
The second equation of 
\eqref{eq:hm-eq-reduced} leads us to 
\begin{equation}\label{eq:solution-g}
\begin{aligned}
  g &= \int \frac{c_1}{\cos^2 f} du = 
c_1 \int \frac{du}{1-a^2 \sn^2 (bu, a/b)} \\
&=\frac{c_1}{b} \varPi (\am(bu, \frac{a}{b}), a^2, \frac{a}{b}), \quad
c_1 = \sqrt{(1-a^2)(b^2-1)} , 
\end{aligned}
\end{equation}
where $\varPi$ denotes the elliptic integral of the third kind 
\begin{equation*}
 \varPi (\varphi, n, k) = \int_0^\varphi 
\frac{d \theta}{(1-n \sin^2 \theta)\sqrt{1-k^2 \sin^2 \theta}} . 
\end{equation*}

The argument up to here proves the following lemma: 
\begin{lemma}\label{lem:gmap-cmc-helicoid}
Let a map $\nbf \colon (U, [du^2+dv^2]) \to S^2$ 
be given in the form \eqref{eq:n:UtoS2} with non-constant $\gamma$. 
Then $\nbf$ is a harmonic map not being \eqref{eq:conf-map-sol}  
if and only if it is given by \eqref{eq:special-sol} or 
\begin{align}
 \nbf &= 
\begin{pmatrix}
 e^{\iii v} & 0 \\ 0 & 1
\end{pmatrix}
\begin{pmatrix}
 e^{\iii g} & 0 \\ 0 & 1
\end{pmatrix}
\begin{pmatrix}
 \sqrt{1-a^2 \sn^2 (bu, \frac{a}{b})} \\ -a \sn(bu, \frac{a}{b})
\end{pmatrix}, \label{al:gmap-cmc-helicoid}
\end{align}
where $a, b$ are constants with $0 < a \le 1 \le b$ and 
$g$ is a function given by \eqref{eq:solution-g}. 
\end{lemma}
We have obtained Lemma \ref{lem:gmap-cmc-helicoid} under the assumption 
that $\gamma$ is non-constant. However, we can include the case that $\gamma$ is 
constant in Lemma \ref{lem:gmap-cmc-helicoid},  
by allowing that the number $a$ can be $0$ in \eqref{al:gmap-cmc-helicoid}. 
Hence, Lemma \ref{lem:gmap-cmc-helicoid}, Proposition \ref{prop:pre-cmc-helicoid} 
and \eqref{eq:c-cylinder-as-heli} 
lead us to the following proposition. 
\begin{proposition}\label{prop:cmc-helicoid}
For constants $a, b$ with $0 \le a \le 1 \le b$, set 
\begin{equation*}
 c_1 = \sqrt{(1-a^2)(b^2-1)}, \ 
g= g(u) =\frac{c_1}{b} \varPi (\am(bu, \frac{a}{b}), a^2, \frac{a}{b}). 
\end{equation*}
Then, for non-zero constant $H$,  
\begin{equation}\label{eq:cmc-helicoid}
-2H \xbf = 
 \begin{pmatrix}
  e^{\iii v} & 0 \\ 0 & 1
 \end{pmatrix}
 \begin{pmatrix}
  e^{\iii g} & 0 \\ 0 & 1
 \end{pmatrix}
\left\{
\begin{pmatrix}
 \sqrt{1-a^2 \sn^2} \\
-a \sn
\end{pmatrix}
+
\begin{pmatrix}
 \dfrac{a( b \cn \dn - \iii c_1 \sn)}{\sqrt{1-a^2 \sn^2}} (bu, \frac{a}{b}) \\
(b^2-1) u -b E (\am (bu, \frac{a}{b}), \frac{a}{b})
\end{pmatrix}
\right\}
+
\begin{pmatrix}
 0 \\ c_1 v
\end{pmatrix}
\end{equation} 
gives a cmc-$H$ helicoid $\xbf \colon U \to \E^3$. 

Conversely, 
any cmc-$H$ helicoid can be parametrized in this manner.   
\end{proposition} 

\begin{remark}\label{rem:nod-undul}
\begin{enumerate}
 \item As do Carmo-Dajczer \cite{doCD} pointed out, 
for a fixed non-zero constant $H$, all cmc-$H$ helicoids form a 
two-parameter family. Proposition \ref{prop:cmc-helicoid} asserts that 
$a, b$ can play a role of parameters of it.  
 \item 
 For the value $a=1$, \eqref{eq:cmc-helicoid} reduces to 
\begin{equation*}
\xbf = 
 \begin{pmatrix}
  e^{\iii v} & 0 \\ 0 & 1
 \end{pmatrix}
\left\{
\begin{pmatrix}
 \cn \\
 \sn
\end{pmatrix}
+
\begin{pmatrix}
  -b \dn   \\
(b-1/b) \text{id} -b E \circ \am
\end{pmatrix}
\right\}(bu, \frac{1}{b}), 
\end{equation*} 
which is a nodoid. 

 For the value $b=1$, \eqref{eq:cmc-helicoid} reduces to 
\begin{equation*}
\xbf = 
 \begin{pmatrix}
  e^{\iii v} & 0 \\ 0 & 1
 \end{pmatrix}
\left\{
\begin{pmatrix}
 \dn \\
 a \sn
\end{pmatrix}
+
\begin{pmatrix}
  -a \cn   \\
 - E \circ \am
\end{pmatrix}
\right\}(u, a), 
\end{equation*} 
which is an unduloid. 
\end{enumerate}
\end{remark}

\subsection{Fundamental forms}
We have obtained an explicit formula for cmc-$H$ helicoids. Now 
we calculate their fundamental forms.

First, we state a general proposition. 
\begin{proposition}\label{prop:ff's-by-n}
Let $\xbf \colon M \to \E^3$ be an immersion which is not a minimal surface.   
Let $\nbf$ be the Gauss map of $\xbf$.  
Then its fundamental forms and Hopf differential are 
given by 
\begin{equation*}
\begin{aligned}
 \fff &= \frac{1}{4 H^2} \left\{
| d \nbf |^2 + 2 | \nbf , \  * d \nbf , \ d \nbf |  + | *d \nbf |^2
\right\},  \\ 
\sff &= \frac{1}{2H} \left\{
|d \nbf|^2 + |\nbf, \ *d\nbf , \ d\nbf |
\right\},  \\ 
\tff &= |d \nbf|^2 , \\ 
 Q &= | (1+\iii *) d \nbf |^2 .  
\end{aligned}
\end{equation*}

Moreover, assume that $\xbf$ is a cmc-$H$ surface.  
Then, for the cgc-companion of $-2 H \xbf$, i.e., 
$\check \xbf = \int \nbf \times (* d \nbf)$, 
the first and second fundamental forms are given by  
\begin{equation*}
 \check \fff  = | * d \nbf |^2, \quad \check \sff = | \nbf, \ d \nbf, \ * d \nbf |,  
\end{equation*}
respectively. 
\end{proposition}
A proof of Proposition \ref{prop:ff's-by-n} 
is straightforward by the definitions $\fff = \la d\xbf, d\xbf \ra$, 
$\sff = -\la d\xbf, d\nbf \ra$, 
$\tff = \la d\nbf, d\nbf \ra$, the formula 
$Q = \sff -H \fff + \iii | \nbf, \ d \xbf, \ * d \nbf |$, 
and \eqref{eq:unified-pre-formula} or \eqref{eq:lelieuvre-type}.  
So the detail is left to the reader.

The following proposition 
follows from Proposition \ref{prop:ff's-by-n} and 
Proposition \ref{prop:cmc-helicoid}. 
\begin{proposition}\label{prop:fff's}
For a cmc-$H$ helicoid $\xbf$ \eqref{eq:cmc-helicoid}, 
  \begin{align*}
  4 H^2 \fff &= \left( (a \cn- b \dn)(bu, \frac{a}{b}) \right)^2 (du^2 +dv^2), \\
4H (\sff - H \fff) &= (c_2-1)(du^2 -dv^2) + 4c_1 dudv , \\
\tff &= (c_2 -a^2 \sn^2(bu, \frac{a}{b}))du^2 +2c_1 dudv 
+(1 -a^2 \sn^2(bu, \frac{a}{b})) dv^2 , \\
Q &= \frac{(\sqrt{b^2-1} -\iii \sqrt{1-a^2} )^2}{4H} (du +\iii dv)^2 
 \end{align*}
and, for the cgc-companion of $-2 H \xbf$,  
\begin{align*}
 \check \fff &= (1 -a^2 \sn^2(bu, \frac{a}{b})) du^2 -2c_1 dudv 
+(c_2 -a^2 \sn^2(bu, \frac{a}{b})) dv^2, \\
 \check \sff &= - ab \cn \dn (bu, \frac{a}{b}) (du^2 + dv^2) ,  
\end{align*}
where $c_1 = \sqrt{(1-a^2)(b^2-1)}$ and $c_2 = a^2 + b^2 -1$. 
\end{proposition}

\begin{remark}[Isothermic nets]
One application of these formulas in Proposition \ref{prop:fff's} 
is that we can obtain an isothermic net on any 
cmc-$H$ helicoidal surface explicitly. The word 
`\textit{isothermic}' means `isothermal and curvature line', in other words, 
the first and second fundamental forms can be simultaneously diagonalized.  
In fact, if we introduce a new parameter $(x, y)$ by 
\begin{equation*}
 \begin{pmatrix}
  x \\ y
 \end{pmatrix}
=
\frac{1}{\sqrt{b^2-a^2}}
\begin{pmatrix}
 \sqrt{1-a^2} & -\sqrt{b^2 -1} \\ 
 \sqrt{b^2 -1} & \sqrt{1-a^2}
\end{pmatrix}
\begin{pmatrix}
 u \\ v
\end{pmatrix}, 
\end{equation*}
then $(x, y)$ is isothermic.  
\end{remark}

\begin{remark}[Cgc-companions are wave fronts.]\label{rem:cgc=wf}
 Another application of Proposition \ref{prop:fff's} is to prove that 
any cgc-companion $\check \xbf$ fails to be an immersion. In fact, one can show that 
the first fundamental form $\check \fff$ degenerates where 
$\sn (bu, \mu) = \pm 1$. This happens for $u = (2k-1) K(a/b) /b$, where 
$k \in \mathbb Z$ and $K(a/b)$ is the complete elliptic integral of the first kind.  
Thus $\check \xbf$ is a wave front because it is a parallel surface of a 
regular surface $\xbf$. 
\end{remark}

\section{Further investigation on cmc-$H$ helicoids}\label{sec:f-i-cmc-heli}
We shall further investigate cmc-$H$ helicoids and their cgc-companions
 based on the explicit expression \eqref{eq:cmc-helicoid}. 

\subsection{Modification}
For the sake of later argument, we start by modifying the formula 
\eqref{eq:cmc-helicoid} slightly. 
We introduce 
\begin{equation*}
 \mu = a/b,   
\end{equation*}
and use $\mu, b$ instead of $a, b$. The range is 
\begin{equation*}
 \mathfrak{M} = \{ (\mu, b) \mid 0 \le \mu \le 1/b, \  1 \le b \}
= \{ (\mu, b) \mid 0 \le \mu \le 1, \  1 \le b \le 1/ \mu \}. 
\end{equation*}
Namely, $\mathfrak M$ can be considered as a parameter space of 
all cmc-$H$ helicoids. $b=1$ and $b=1/\mu$ correspond to 
unduloids and nodoids, respectively. In particular,   
$(\mu, b) =(1,1)$ corresponds to the round sphere. 
$\mu =0$ with arbitrary $b$ corresponds to the same circular cylinder. 
Note that a circular cylinder can be considered as a helicoidal surface whose 
value of pitch is arbitrary.

The first fundamental form $\fff$ and Hopf differential $Q$ are 
rewritten as 
\begin{align*}
 4 H^2 \fff &= \{ (\mu \cn - \dn)(bu, \mu) \}^2 b^2 (du^2 +dv^2), \\
 4H Q &= (1-\mu^2) e^{\iii \theta(\mu, b)} b^2 (du+\iii dv)^2 , 
\end{align*}
where $\theta =\theta(\mu, b)$ is determined by 
\begin{equation*}
 \cos \theta = \sqrt{\frac{b^2-1}{b^2(1-\mu^2)}}, \quad
 \sin \theta = -\sqrt{\frac{1-\mu^2 b^2}{b^2(1-\mu^2)}}. 
\end{equation*}
Setting $bu =\tilde u$, $bv = \tilde v$, we have 
\begin{equation}\label{eq:fff-Q-cmchelicoid}
\begin{aligned}
 4 H^2 \fff &= \{ (\mu \cn - \dn)(\tilde u, \mu) \}^2
 (d \tilde u^2 +d \tilde v^2),  \\
 4H Q &= (1-\mu^2) e^{\iii \theta(\mu, b)} 
 (d \tilde u+\iii d \tilde v)^2 . 
\end{aligned}
\end{equation}
 The formulas \eqref{eq:fff-Q-cmchelicoid} imply that two cmc-$H$ helicoids 
having the same $\mu$ and distinct $b$ are isometric but non-congruent. 
In other words, two cmc-$H$ helicoids belong to the same associated family 
if and only if they have the same value of $\mu$. 
Moreover, for any fixed $\mu$, the family $\{ \xbf \mid 1 \le b \le 1/\mu \}$ 
includes an unduloid ($b=1$) and a nodoid ($b=1/\mu$).  

We modify Proposition \ref{prop:cmc-helicoid} to a more convenient statement  
by replacing $(bu,bv)$ by $(u,v)$ and $a/b$ by $\mu$. 

\begin{proposition}\label{prop:cmc-helicoid-rev}
Let $(\mu, b) \in \mathfrak M = 
\{ (\mu, b) \mid \ 0 \le \mu \le 1, \ 1 \le b \le 1/\mu \}$ 
and set 
\begin{align*}
 c_1 &= \sqrt{(1-\mu^2 b^2)(b^2-1)}, \\
 \nbf_0(u) &= 
\begin{pmatrix}
 \sqrt{1-\mu^2 b^2 \sn^2} \\
-\mu b \sn
\end{pmatrix}(u, \mu), \\ 
 \check \xbf_0 (u) &= 
\begin{pmatrix}
 \dfrac{\mu b( b \cn \dn - \iii c_1 \sn)}{\sqrt{1-\mu^2 b^2 \sn^2}} (u, \mu) \\
(b-1/b) u -b E (\am (u, \mu), \mu)
\end{pmatrix} ,  \\ 
 g(u) &=\frac{c_1}{b} \varPi (\am(u, \mu), \mu^2 b^2, \mu).  
\end{align*}
Then, for a non-zero constant $H$,     
\begin{equation}\label{eq:cmc-helicoid-rev}
-2H \xbf = 
 \begin{pmatrix}
  e^{\iii v/b} & 0 \\ 0 & 1
 \end{pmatrix}
 \begin{pmatrix}
  e^{\iii g(u)} & 0 \\ 0 & 1
 \end{pmatrix}
\left\{
\nbf_0(u)+ \check \xbf_0 (u) 
\right\}
+
\begin{pmatrix}
 0 \\ c_1 v/b
\end{pmatrix}
\end{equation} 
gives a cmc-$H$ helicoid $\xbf \colon U \to \E^3$. 

Conversely, 
any cmc-$H$ helicoid can be parametrized in this manner. 
\end{proposition}

\subsection{Boundedness by circular cylinders}\label{subsec:bddness}
In this paper, we say that a helicoidal surface $\mathcal H$ is 
{\it bounded} (or {\it bounded outward}) 
if there exists a circular cylinder $\mathcal C$
such that $\mathcal H$ is included inside of $\mathcal C$.  Similarly, 
 $\mathcal H$ is said to be {\it bounded inward} 
if there exists a circular cylinder $\mathcal C$
such that $\mathcal H$ is included outside of $\mathcal C$. 

Thanks to the formula \eqref{eq:cmc-helicoid-rev}, we can easily estimate the 
boundedness of cmc-$H$ helicoids. 

\begin{proposition}
Any helicoidal surface of positive constant Gaussian curvature is bounded. 
It is also bounded inward except for the parallel surface of an unduloid. 
\end{proposition}
\begin{proof}
We may suppose a helicoidal surface of positive constant Gaussian curvature 
is given by the cgc-companion $\check \xbf$ of \eqref{eq:cmc-helicoid-rev}. 
We have only to prove the assertion for a special value of $H$, hence 
we assume $H = -1/2$ here. 

 For $\check \xbf = {}^t(x,y,z)$, set $p(\check \xbf) = {}^t(x,y,0)$. Then, 
by straightforward computation, we have  
\begin{align*}
 |p(\check \xbf)|^2 &= 
\left|
 \dfrac{\mu b( b \cn \dn - \iii c_1 \sn)}{\sqrt{1-\mu^2 b^2 \sn^2}} (u, \mu) 
\right|^2 \\
&= 
\mu^2 b^2 (b^2 - \sn^2(u, \mu) ). 
\end{align*}
Since $0 \le \sn^2(u, \mu) \le 1$, we have 
\begin{equation}\label{ineq:cgc}
 \mu b \sqrt{b^2-1} \le |p(\check \xbf)| \le \mu b^2.
\end{equation}
We may exclude the case $\mu =0$ where 
the image of $\check \xbf$ degenerates to a line. 
Thus, unless $\mu=0$, the image of $\check \xbf$ is 
inside a circular cylinder of radius $\mu b^2$, and 
outside a circular cylinder of radius $\mu b \sqrt{b^2 -1}$ if $b \ne 1$. 

If $b =1$, then the inner radius equals 0, indeed, 
$\xbf$ is an unduloid. 
\end{proof}
For a cgc-companion with a general value $H$ ($\ne 0$), 
the inequality \eqref{ineq:cgc} is 
\begin{equation*}
 \mu b \sqrt{b^2-1}/|2H| \le |p(\check \xbf)| \le \mu b^2/|2H|, 
\end{equation*}
that is, 
\begin{equation}\label{ineq:cgc-gen}
 \mu b \sqrt{b^2-1}/\sqrt{K} \le |p(\check \xbf)| \le \mu b^2/\sqrt{K}. 
\end{equation}

In \eqref{ineq:cgc-gen}, the equality can happen. 
We call $\mu b \sqrt{b^2-1}/\sqrt{K}$, $\mu b^2/\sqrt{K}$
 the \textit{inner radius}, \textit{outer radius} of $\check \xbf$, respectively. 
The outer radius is simply called the \textit{radius} 
as long as there is no confusion. 
\begin{proposition}[cf. Remark 4.15 of \cite{doCD}]
Any cmc-$H$ helicoid is bounded.  
\end{proposition}
\begin{proof}
We may assume that a cmc-$H$ helicoid $\xbf$ is given by \eqref{eq:cmc-helicoid-rev}. 
We have only to prove the assertion for a special value of $H$, hence 
we assume $H = -1/2$ here. 

 For $\xbf = {}^t(x,y,z)$, set $p(\xbf) = {}^t(x,y,0)$. Then, 
by straightforward computation, we have  
\begin{align*}
 |p(\xbf)|^2 &= 
\left|
\sqrt{1-\mu^2 b^2 \sn^2(u,\mu)} + 
 \dfrac{\mu b( b \cn \dn - \iii c_1 \sn)}{\sqrt{1-\mu^2 b^2 \sn^2}} (u, \mu) 
\right|^2 \\
&= 
1+ \mu^2 b^4 - 2 \mu^2 b^2 \sn^2(u, \mu) -2 \mu b^2 \cn (u, \mu) \dn(u, \mu) \\ 
&= 
1+ \mu^2 b^4 - 2 \mu^2 b^2 + 2 \mu b^2 \{ 
\mu t^2 + t \sqrt{(1-\mu^2) + \mu^2 t^2 } \}. 
\end{align*}
where we set $t = \cn(u,\mu)$. 
On the other hand, one can verify  
\begin{equation*}
 \mu-1 \le \mu t^2 + t \sqrt{(1-\mu^2) + \mu^2 t^2} \le \mu + 1 
\quad \text{if} \quad 
0 \le \mu \le 1, -1 \le t \le 1. 
\end{equation*} 
Therefore, 
\begin{align*}
 1+\mu^2 b^4 -2 \mu^2 b^2 +2 \mu b^2 (\mu -1)
\le |p(\xbf)|^2 \le 
1+\mu^2 b^4 -2 \mu^2 b^2 +2 \mu b^2 (\mu +1) , 
\end{align*}
Hence
\begin{equation*}
 (1 - \mu b^2)^2 \le |p(\xbf)|^2 \le (1 + \mu b^2)^2 , 
\end{equation*}
i.e., 
\begin{align}\label{ineq:cmc}
| 1 - \mu b^2 | \le |p(\xbf)| \le 1+\mu b^2. 
\end{align}
Thus, $\xbf$ is 
inside a circular cylinder of radius $1+\mu b^2$ and 
outside a circular cylinder of radius $|1-\mu b^2|$.
\end{proof}
For a cmc-$H$ helicoid with a general value $H$ ($\ne 0$), 
the inequality \eqref{ineq:cmc} is 
\begin{align}\label{ineq:cmc-gen}
| 1 - \mu b^2 |/|2H| \le |p(\xbf)| \le (1+\mu b^2)/|2H|. 
\end{align}

Note that the equalities in \eqref{ineq:cmc-gen} can happen. 
We call $|1- \mu b^2|/|2H|$, $1+\mu b^2/|2H|$
the \textit{inner radius}, \textit{outer radius} of $\xbf$, respectively. 
The outer radius is simply called the \textit{radius} 
as long as there is no confusion. 

In \cite{doCD}, the inequality \eqref{ineq:cmc-gen} was already pointed out, 
however it was not fully discussed there whether $\xbf$ can have 
zero inner radius (see Remark 4.15 in \cite{doCD}). 
The inequality \eqref{ineq:cmc-gen} tells us that 
 a cmc-$H$ helicoid $\xbf$ has inner radius 0 if and only if $\mu b^2 = 1$. 
In other words, a cmc-$H$ helicoid $\xbf$ is bounded inward 
if and only if $\mu b^2 \ne 1$.   
Thus, a cmc-$H$ helicoid $\xbf$ is bounded both inward and outward, 
unless $\mu b^2 =1$. 
We also note that the case where $\mu b^2 =1$ does actually occurs. In fact, 
we can assert the following: 
\begin{proposition}
 Except for the circular cylinders, 
each associated family includes a cmc-$H$ helicoid with zero inner radius. 
\end{proposition}
\begin{proof}
 Let $\xbf$ be a cmc-$H$ helicoid determined by $(\mu,b) \in \mathfrak M$, 
and assume $\xbf$ is not a circular cylinder, i.e., $\mu \ne 0$.  
If we fix $\mu$, then there exists a unique $b$ 
such that $(\mu,b) \in \mathfrak M$ and $\mu b^2 =1$. 
\end{proof}
If a cmc-$H$ helicoid $\xbf$ has zero inner radius, then there exists a point on 
$\xbf(M) \cap \text{$z$-axis}$. The trajectory of such a point by a helicoidal 
motion forms a line coincident with $z$-axis. Thus we can assert as follows:   
\begin{corollary}
 Each associated family includes a cmc-$H$ helicoid 
on which a straight line lies. 
\end{corollary}
In Figures \ref{fig:zeroi-a} and \ref{fig:zeroi-b},   
cmc-$H$ helicoids on which a line (the axis of helicoidal motions) lies
are shown.   
The cmc-$H$ helicoid in Figure \ref{fig:zeroi-b} has a period 
(periods will be explained in the next subsection).  
\begin{figure}[htbp]
\begin{tabular}{cc}
\begin{minipage}{0.5\hsize}
\centering
\includegraphics[width=4cm,clip]{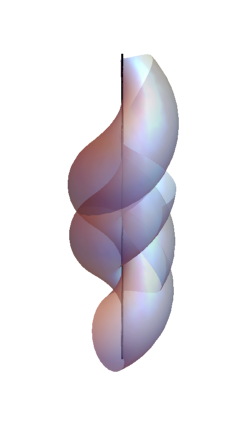}  
\caption{ 
}
\label{fig:zeroi-a}
\end{minipage}
\quad
\begin{minipage}{0.5\hsize}
\centering
\includegraphics[width=4cm,clip]{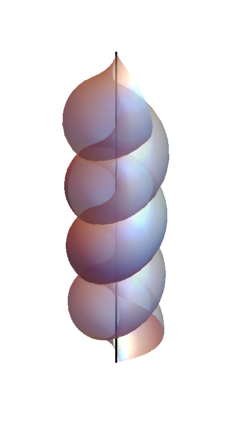} 
\caption{}
\label{fig:zeroi-b}
\end{minipage}
\end{tabular}
\end{figure}

\subsection{Cmc helicoidal cylinders}
It can happen that a helicoidal surface $\xbf$ has a period, that is, 
it reduces to a map from $S^1 \times \R$ to $\E^3$. 
A helicoidal surface having a period is called a \textit{helicoidal cylinder}. 
We show that both cmc-$H$ helicoids and their cgc-companions 
can have periods in some cases.  

In this section, let $H =-1/2$, and let 
\begin{equation}\label{eq:complete-ei}
K:=K(\mu),\quad E:=E(\mu), \quad \varPi := \varPi(\mu^2 b^2, \mu)  
\end{equation}
denote the complete elliptic integrals 
of the first, second and third kind, respectively.

\begin{lemma}\label{lem:quasi-periodic}
For any fixed $(\mu , b) \in \mathfrak M$, 
the following equations of quasi-periodicity hold: 
\begin{align*}
g(u+2K) &= g(u) + \frac{2c_1}{b} \varPi , \\
 \nbf_{0}(u +2K) &= 
\begin{pmatrix}
 1 & 0 \\ 0 & -1
\end{pmatrix}
 \nbf_{0}(u) , \\ 
\check \xbf_{0}(u+2K) &=
\begin{pmatrix}
 -1 & 0 \\ 0 & 1
\end{pmatrix}
\check \xbf_{0}(u)
+
\begin{pmatrix}
0 \\
2(b-\frac{1}{b}) K -2 b E
\end{pmatrix}. 
\end{align*}
\end{lemma}
A proof of Lemma \ref{lem:quasi-periodic} 
is straightforward by the definition of $g, \nbf_0 , \check \xbf_0$ 
(cf. Proposition \ref{prop:cmc-helicoid-rev}) and the (quasi-)periodicity of 
elliptic integrals and Jacobi's elliptic functions. 
So the proof is omitted here and is left to the reader. 
\begin{proposition}\label{prop:period-cond}
Let $\check \xbf$ be a cgc-companion of a cmc-$H$ helicoid 
determined by $(\mu,b) \in \mathfrak M$. Then 
$\check \xbf$ 
is periodic if and only if 
\begin{equation}\label{eq:peri-cond-cgc}
 \Phi(\mu,b) :=\frac{c_1^2 \varPi +b^2E +(1-b^2)K}{\pi b c_1} \in \mathbb Q, 
\end{equation}
where $c_1=\sqrt{(1-\mu^2b^2)(b^2-1)}$ and $K, E, \varPi$ are 
complete elliptic integrals \eqref{eq:complete-ei}. 
\end{proposition}
\begin{proof}
 It follows from Proposition \ref{prop:cmc-helicoid-rev} and Lemma \ref{lem:quasi-periodic} 
that 
\begin{align*}
 & \check \xbf(u+2mK,v+h) \\
=& 
\begin{pmatrix}
 e^{\iii(2 m  c_1 \varPi +  h)/b + m \iii \pi} & 0 \\ 0 & 1
\end{pmatrix}
\check \xbf(u,v) 
+ 
\begin{pmatrix}
 0 \\ c_1 h/b + 2m \{ (b-\frac{1}{b})K -b E \}
\end{pmatrix}
\end{align*}
for $m \in \mathbb Z$ and $h \in \R$. Hence, if we 
were able to choose $m, h$ so that 
\begin{equation}\label{eq:per-cond-cgc-helicoid-0}
\begin{cases}
\frac{1}{b}(2 m  c_1 \varPi +  h) + m \pi \in 2 \pi \mathbb Z \\ 
2m \{ (b-\frac{1}{b})K -b E \} +  \frac{c_1 h}{b} =0
\end{cases}
 \end{equation}
then 
\begin{equation*}
 \check \xbf(u+2mK,v+h) = \check \xbf (u,v)
\end{equation*}
holds. 

Assume first that there exist $m, h$ satisfying \eqref{eq:per-cond-cgc-helicoid-0}. Then 
\begin{equation}\label{eq:h=}
 h = -\frac{2m}{c_1} \{ (b^2-1)K -b^2E \}
= \frac{2m}{c_1} \{b^2E + (1 - b^2)K \}.   
\end{equation}
Substituting \eqref{eq:h=} into the first equation 
of \eqref{eq:per-cond-cgc-helicoid-0}, 
we have 
\begin{equation*}
m(\frac{c_1}{b} \varPi + \frac{\pi}{2}) 
+\frac{m}{b c_1} \{b^2E+ (1-b^2)K \}
  \in \pi \mathbb Z,  
\end{equation*}
that is, 
\begin{equation*}
 m \pi \Phi(\mu,b) + \frac{m}{2} \pi \in \pi \mathbb Z. 
\end{equation*}
Therefore the condition \eqref{eq:peri-cond-cgc} holds.

Conversely we suppose \eqref{eq:peri-cond-cgc}. Then 
$\Phi(\mu,b) + \frac{1}{2} \in \mathbb Q$, that is,   
there exist mutually prime integers 
$p, q$ such that     
\begin{equation*}
 \Phi(\mu,b) + \frac{1}{2} = \frac{q}{p}, 
\end{equation*}
namely, 
\begin{equation*}
( \frac{c_1}{b} \varPi + \frac{\pi}{2}) 
+ \frac{1}{b c_1} \{b^2E + (1-b^2)K \} = \frac{q}{p} \pi.   
\end{equation*}
Hence, 
\begin{equation*}
p ( \frac{c_1}{b} \varPi + \frac{\pi}{2}) 
+\frac{p}{b c_1} \{b^2 E + (1-b^2)K \} \in  \pi \mathbb Z.
\end{equation*}
Therefore, 
\begin{equation}\label{eq:m-h}
 m:=p, \quad h := \frac{2 p}{c_1} \{b^2 E + (1-b^2)K \}
\end{equation}
satisfy \eqref{eq:per-cond-cgc-helicoid-0}. 
\end{proof}

\begin{corollary}
There exist infinitely many non-congruent helicoidal cylinders 
with constant positive Gaussian curvature, 
i.e.,  periodic helicoidal wave fronts with constant positive Gaussian curvature.  
\end{corollary}
\begin{proof}
The function $\Phi(\mu, b)$ is non-constant 
and real-analytic in $(\mu, b)$. 
Hence, there exist infinitely many $(\mu, b)$ such that 
$\Phi(\mu, b) \in \mathbb Q$. 
\end{proof}

From now on, by a \textit{helicoidal cgc-cylinder}, 
we mean a helicoidal cylinder with positive constant Gaussian curvature. 

When \eqref{eq:peri-cond-cgc} is satisfied, there exist 
$m \in \mathbb Z$ and $h \in \R$ 
such that 
\begin{equation}\label{eq:per-cgc-co}
 \check \xbf (u+2m K, v+h) = \check \xbf (u,v)
\end{equation}
holds for any $u, v$. We always choose $m, h$ so that $m$ is the minimum 
positive integer. Then $m$ equals the number of cuspidal edges of $\check \xbf$ 
(cf. Remark \ref{rem:cgc=wf}). 
Thus $\check \xbf$ is (non-)co-orientable if $m$ is even (odd). 
We can verify it precisely as follows: 
At the same time to \eqref{eq:per-cgc-co}, 
\begin{equation*}
 \nbf (u+2m K, v+h) = (-1)^m \nbf (u,v). 
\end{equation*}
This implies that if $m$ is even, then $\nbf$ has the same period as $\check \xbf$, 
but if $m$ is odd, then $\nbf$ reverse its direction. 
Namely, if $m$ is chosen to be an odd number then $\check \xbf$ is non-co-orientable. 
(A wave front is said to be \textit{co-orientable} if it posses 
a global unit normal field. For details, see \cite{KU}, etc.)
Hence we can assert that there are infinitely many 
non-co-orientable helicoidal cgc-cylinders. 

\begin{example}
 Let $\mu = 1/2$. The graph of the function $b \mapsto \Phi(1/2, b)+1/2$ is  
as in Figure \ref{fig:graph}. 
\begin{figure}[htbp]
\centering
\includegraphics[width=8cm,clip]{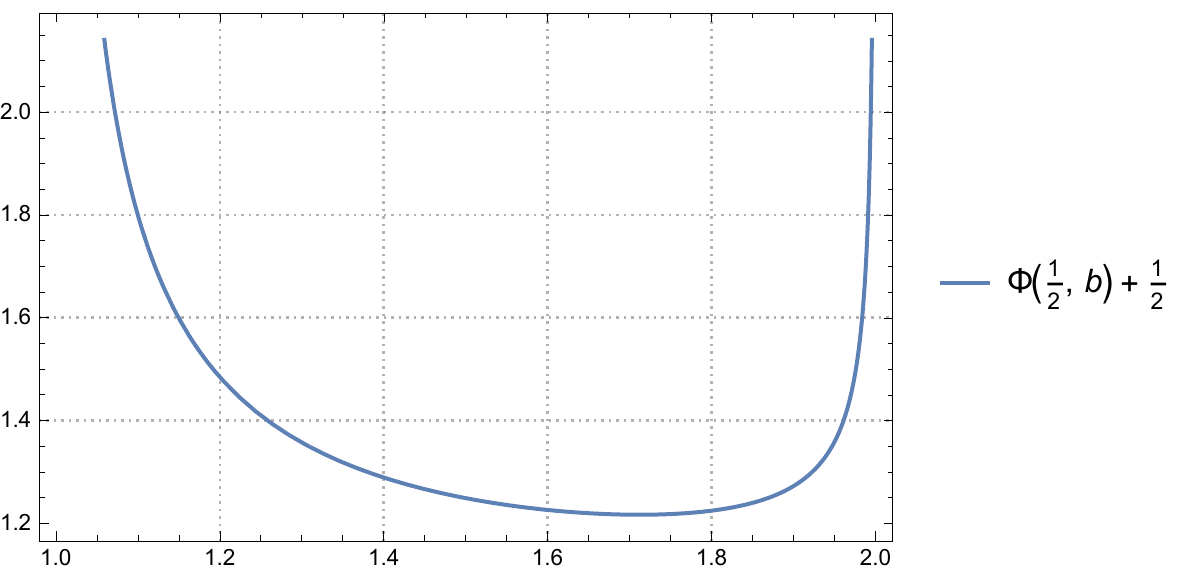} 
\caption{}
\label{fig:graph}
\end{figure}

First, we consider an equation $\Phi(1/2, b)+1/2 = 2 (=2/1)$. This equation 
has two distinct solutions $b =b_1, b_2$.  
They are approximately $b_1 \approx 1.07213$, $b_2 \approx 1.99434$.  

For $b_1$, the formula \eqref{eq:m-h} indicates $m=1$, $h =h_1 \approx 8.7932$ 
which satisfy \eqref{eq:per-cond-cgc-helicoid-0}.  
For $b_2$, the formula \eqref{eq:m-h} indicates $m=1$, $h =h_2 \approx 12.6016$ 
which satisfy \eqref{eq:per-cond-cgc-helicoid-0}.  
Since $m=1$, they are non-co-orientable helicoidal cgc-cylinders having 
a single cuspidal edge. Indeed, 
graphics of helicoidal cgc-cylinders of $(\mu,b) = (1/2,b_1), (1/2, b_2)$ are as 
in Figures \ref{fig:cgc-a}, \ref{fig:cgc-b}. 
\begin{figure}[htbp]
\begin{tabular}{cc}
\begin{minipage}{0.45\hsize}
\centering
 \includegraphics[width=5cm,clip]{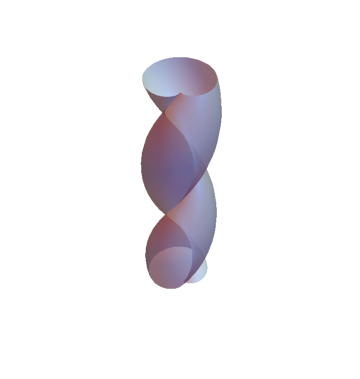}  
\caption{$(\mu,b) = (1/2,b_1)$}
\label{fig:cgc-a}
\end{minipage}
\quad
\begin{minipage}{0.45\hsize}
\centering
 \includegraphics[width=5cm,clip]{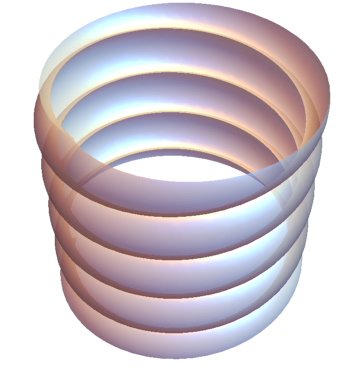} 
\caption{$(\mu,b) = (1/2, b_2)$}
\label{fig:cgc-b}
\end{minipage}
\end{tabular}
\end{figure}

In the next place, we consider an equation $\Phi(1/2, b)+1/2 = 3/2$. 
This equation has two distinct solutions $b =b_1', b_2'$.  
They are approximately $b_1' \approx 1.19174$, $b_2' \approx 1.97619$.

For $b_1'$, the formula \eqref{eq:m-h} indicates $m=2$, $h =h_1' \approx 10.57012$ 
which satisfy \eqref{eq:per-cond-cgc-helicoid-0}.  
For $b_2'$, the formula \eqref{eq:m-h} indicates $m=2$, $h =h_2' \approx 12.70952$ 
which satisfy \eqref{eq:per-cond-cgc-helicoid-0}.  
Since $m=2$, they are co-orientable helicoidal cgc-cylinders 
having two cuspidal edges. Indeed, 
graphics of helicoidal cgc-cylinders of $(\mu,b) = (1/2,b_1'), (1/2, b_2')$ are as 
in Figures \ref{fig:cgc2-a}, \ref{fig:cgc2-b}. 
\begin{figure}[htbp]
\begin{tabular}{cc}
\begin{minipage}{0.45\hsize}
\centering
\includegraphics[width=5cm,clip]{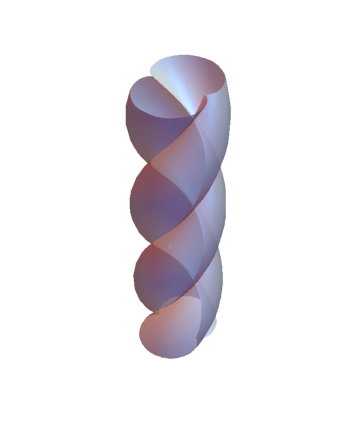} 
\caption{$(\mu,b) = (1/2,b_1')$}
\label{fig:cgc2-a}
\end{minipage}
\quad
\begin{minipage}{0.45\hsize}
\centering
\includegraphics[width=5cm,clip]{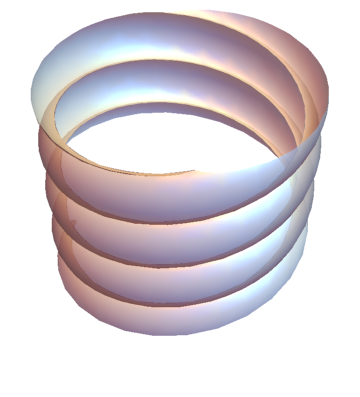} 
\caption{$(\mu,b) = (1/2, b_2')$}
\label{fig:cgc2-b}
\end{minipage}
\end{tabular}
\end{figure}
\end{example}

Finally, we note that the condition \eqref{eq:peri-cond-cgc} is also 
the period condition for cmc-$H$ helicoids (not only for cgc-companion).
Indeed, 
there exist $m \in \mathbb Z$ and $h \in \R$ 
such that 
\begin{equation*}
  \xbf (u+2m K, v+h) = \check \xbf (u,v) + (-1)^m \nbf (u,v) 
\end{equation*}
when \eqref{eq:peri-cond-cgc} is satisfied. Therefore, 
\begin{align*}
 \xbf (u+2m K, v+h) &= \check \xbf (u,v) + \nbf (u,v) = \xbf (u,v) \text{ if $m$ is even},  \\
\xbf (u+4m K, v+h) &= \xbf (u,v) \text{ if $m$ is odd.}
\end{align*}
Thus $\xbf$ is periodic, i.e., $\xbf$ is a helicoidal cylinder 
with constant mean curvature.  
For any fixed $\mu$, there are infinitely many values of $b$ such that 
\eqref{eq:peri-cond-cgc} is satisfied. Thus   
we have given another proof of the following theorem due to Burstall-Kilian.
\begin{theorem}[Theorem 7.6 in \cite{BK}]
  In each associated family of a Delaunay surface, there are infinitely many 
non-congruent cylinders with screw-motion symmetry.
\end{theorem}

Graphics in Figures \ref{fig:cmc2-a}, \ref{fig:cmc2-b} are 
cmc-$H$ helicoidal cylinders determined by 
$\mu =1/2$ and $\Phi(1/2, b) +1/2 = 3/2$, that is, 
parallel cmc-$H$ surfaces of those in Figures \ref{fig:cgc2-a}, \ref{fig:cgc2-b}. 

\begin{figure}[htbp]
\begin{tabular}{cc}
\begin{minipage}{0.45\hsize}
\centering
\includegraphics[width=4cm,clip]{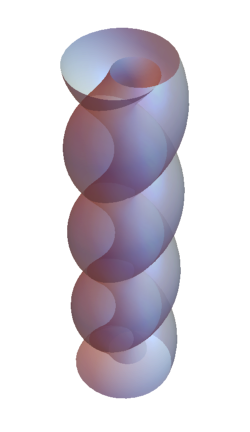}  
\caption{$(\mu,b) = (1/2,b_1')$}
\label{fig:cmc2-a}
\end{minipage}
\quad
\begin{minipage}{0.45\hsize}
\centering
\includegraphics[width=4cm,clip]{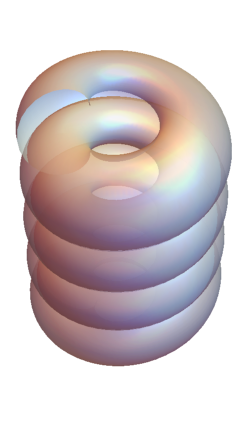} 
\caption{$(\mu,b) = (1/2, b_2')$}
\label{fig:cmc2-b}
\end{minipage}
\end{tabular}
\end{figure}

\subsection{Geometric interpretation}\label{subsec:g-i}
Without loss of generality, we fix $H = -1/2$ in this subsection. 

We have seen that a cmc-$H$ helicoid (or equivalently a cgc-companion) 
is determined by two parameters $\mu, b$. We can restate this so that 
{\it a cmc-$H$ helicoid (or equivalently a cgc-companion) 
is determined by the pitch and radius.} In fact, 
for a cmc-$H$ helicoid $\xbf$ determined by $(\mu,b) \in \mathfrak M$, 
the pitch $\lambda$ and the outer radius $R$ are given by 
\begin{equation}\label{eq:p-and-r}
 \lambda (=c_1) = \sqrt{(1-\mu^2 b^2)(b^2-1)}, \quad R = 1+ \mu b^2. 
\end{equation}    
\eqref{eq:p-and-r} gives a bijective correspondence 
between $\mathfrak M$ and 
\begin{equation*}
 \mathfrak L = \{ (\lambda, R) \mid 0 \le \lambda < \infty, \ 1 \le R < \infty \},  
\end{equation*}
whose inverse is 
\begin{equation}\label{eq:inverse}
 \mu = \frac{\sqrt{\lambda^2 +R^2} - \sqrt{\lambda^2 +(R-2)^2}}%
{\sqrt{\lambda^2 + R^2} + \sqrt{\lambda^2 +(R-2)^2}} , \quad 
b = \frac{1}{2} \left\{ \sqrt{\lambda^2 +R^2} + \sqrt{\lambda^2 +(R-2)^2} \right\}. 
\end{equation}

Moreove we have the following assertion. 
\begin{theorem}
It is determined by the pitch and radius whether a cmc-$H$ helicoid has a period or not. 
More precisely, 
a cmc-$H$ helicoid of the pitch $\lambda$ and the radius $R$ has a period if and only if 
the value of $\Phi=\Phi(\mu, b)$ substituted with \eqref{eq:inverse} 
is rational.  
\end{theorem}

Let $\rho$ denote the inner radius of $\xbf$. Then $\rho = |1 - \mu b^2| = |2-R|$. 
Thus the first equation of \eqref{eq:inverse} is written as 
\begin{equation}\label{eq:muu}
  \mu = \frac{\sqrt{\lambda^2 +R^2} - \sqrt{\lambda^2 +\rho^2}}%
{\sqrt{\lambda^2 + R^2} + \sqrt{\lambda^2 +\rho^2}}
\end{equation}
 An application of \eqref{eq:muu} is the following. 
\begin{theorem}
 Two cmc-$H$ helicoids belong to the same associated family if and only if 
the values of 
\begin{equation*}
 \frac{\lambda^2 + \rho^2}{\lambda^2 + R^2} 
\left( = 
 \frac{\lambda^2 + (R-2)^2}{\lambda^2 + R^2} 
 \right)
\end{equation*}
are coincident. 
\end{theorem}
\begin{proof}
Two cmc-$H$ helicoids belong to the same associated family if and only if 
the value of $\mu \in [0,1]$ are coincident. On the other hand, 
 \eqref{eq:muu} can be rewritten as 
\begin{equation*}
 \frac{\sqrt{\lambda^2 +\rho^2}}{\sqrt{\lambda^2 +R^2}}
= \frac{1-\mu}{1+\mu} .
\end{equation*}
\end{proof}
\begin{corollary}
 An unduloid and a nodoid are associated (i.e., locally isometric) 
if and only if the ratios 
\begin{equation*}
 \frac{\rho}{R}
\end{equation*}
of the inner radius $\rho$ and the outer radius $R$ are coincident. 
\end{corollary}

\section*{Acknowledgments}

The author thanks Wayne Rossman for his pertinent advice.  
He also thanks the members of Geometry Seminar at Tokyo Institute of 
Technology for giving him the opportunity to talk 
and for inspiring him to improve 
the section \ref{sec:f-i-cmc-heli}.

\bibliographystyle{amsplain}

\end{document}